\documentclass[lang=en]{elegantpaper}
\usepackage{color}
\usepackage{enumitem}
\usepackage{hyperref}
\usepackage{tikz}
\usepackage{tikz-cd}
\usetikzlibrary{arrows.meta,decorations.markings}

\usepackage{amsthm}
\theoremstyle{plain}
\newtheorem{numberedcorollary}[theorem]{Corollary}
\usepackage{cleveref}

\NewDocumentCommand{\inorm}{o m}{
	\IfValueTF{#1}{
		\mathrm{N}(\mathfrak{#2})^{#1}
	}{
		\mathrm{N}(\mathfrak{#2})
	}
}
\NewDocumentCommand{\iphi}{ e{^} m }{%
  \IfValueTF{#1}{%
	\phi^*(\mathfrak{#2})
  }{%
    \phi(\mathfrak{#2})%
  }%
}
\newcommand{\imu}[1]{\mu(\mathfrak{#1})}
\newcommand{\mmu}[1]{\widetilde{\mu}(\mathfrak{#1})}

\usepackage{xparse}
\usepackage{mathtools}

\NewDocumentCommand{\ssum}{
	e{_}
}{
	\IfValueTF{#1} {
		\sum_{\substack{#1}}
	}{
		\sum
	}
}

\renewcommand{\mod}[1]{{\ifmmode\text{\rm\ (mod~$#1$)}\else\discretionary{}{}{\hbox{ }}\rm(mod~$#1$)\fi}}

\newcommand{\bigO}{\mathsf{O}}
\newcommand{\littleo}{\mathsf{o}}

\def\a{\mathfrak{a}}
\def\f{\mathfrak{f}}
\def\p{\mathfrak{p}}
\def\a{\mathfrak{a}}
\def\m{\mathfrak{m}}
\def\s{\mathfrak{s}}
\def\t{\mathfrak{t}}

\def\norm{\mathrm{N}}

\def\to{\rightarrow}

\def\1{1\!\!1}

\DeclareMathOperator{\Tr}{Tr}

\title{Nonvanishing of ray class $L$-functions}
\author{Thurman Ye\qquad Xu Zhuang}
\date{}

\begin{document}
\maketitle
\begin{abstract}
We study the nonvanishing of central values of ray class \(L\)-functions
over a fixed imaginary quadratic field \(K\). We prove that, as
\(\norm(\mathfrak f)\) tends to infinity, at least a proportion
\(1/3-\littleo_K(1)\) of the primitive ray class characters \(\chi\)
modulo \(\mathfrak f\) satisfy \(L(1/2,\chi)\neq 0\). Our proof adapts
the classical mollifier method to the ray class setting and establishes
asymptotic formulas for the first and second mollified moments.
\end{abstract}
\section{Introduction}

The nonvanishing of the central values of \(L\)-functions is of great importance in number theory. For the classical case, Chowla conjectured in \cite{Chowla} that \( L(1/2, \chi) \neq 0 \) over all primitive Dirichlet characters \(\chi\). Balasubramanian and Murty \cite[p.~568]{BalasubramanianMurty1992} proved the first positive proportion of nonvanishing for Dirichlet characters;
Iwaniec and Sarnak \cite{Iwaniec-Sarnak} significantly improved the proportion to \(1/3\). They obtained their results by computing the first two mollified moments of \(L(1/2,\chi)\). Subsequent work has raised the proportion through a combination of more sophisticated mollifiers, sharper estimates for the corresponding moments, and restrictions to certain moduli. For a partial survey of these results, see \cite{Bui2012},  \cite{QinWu2025}, \cite{KhanMilicevicNgo2022}, and \cite{durkan2026centralnonvanishingdirichletlfunctions}.

In this paper we study the nonvanishing problem for characters over number fields. Fix a number field \(K\) and integral ideal \(\mathfrak{f}\), and consider \(\operatorname{Cl}_K(\mathfrak f)\),  the ray
class group modulo \(\mathfrak f\). Associated to
a character \(\chi\) of \(\operatorname{Cl}_K(\mathfrak f)\) is the Hecke
\(L\)-function \( L(s, \chi) \). Assuming that \(K\) is imaginary quadratic, our main result is to recover the Iwaniec-Sarnak proportion of nonvanishing of \(L(\frac{1}{2},\chi) \) in the limit \( \norm(\mathfrak{f}) \to \infty \).

\begin{theorem}\label{thm:main}
Let \(K\) be an imaginary quadratic field. As \( \norm(\mathfrak{f}) \to \infty \) through integral ideals \( \mathfrak{f} \), at least \( (\frac{1}{3} - \mathsf{o}_K(1))h_K^*(\mathfrak f)\) primitive ray class characters modulo \(\mathfrak{f} \) satisfy \( L(\frac{1}{2}, \chi) \neq 0 \).
\end{theorem}

Our work adapts the one-piece mollifier method of Iwaniec and Sarnak to the imaginary quadratic setting. Gao and Zhao \cite{GaoZhao2020QuadraticHecke} obtained quantitative nonvanishing results for quadratic Hecke \(L\)-functions over \(\mathbb Q(i)\) and \(\mathbb Q(\sqrt{-3})\). More recently, David, de Faveri, Dunn, and Stucky \cite{david2026nonvanishingcubicheckelfunctions} proved unconditional positive-proportion nonvanishing for a family of cubic Hecke \(L\)-functions over \(\mathbb Q(\sqrt{-3})\). Castillo, de Faveri, and Dunn \cite{castillo2026nonvanishingquarticheckelfunctions} established an analogous result for a family of quartic Hecke \(L\)-functions over \(\mathbb Q(i)\), and also treated twists of fixed infinity type. Diaconu, Ion, Pa\c{s}ol, and Popa \cite{diaconu2026secondmomentnonvanishingcentral} established unconditional positive-proportion nonvanishing for families of \(r\)-th order Hecke \(L\)-functions over number fields containing the \(2r\)-th roots of unity, for each fixed \(r\ge3\). Their results concern families indexed by squarefree and \(r\)-th power-free ideals. These works study prescribed families of fixed-order characters or their twists as the indexing ideals vary, whereas we average over all primitive ray class characters modulo a given ideal, without restricting their order. The resulting nonvanishing proportions are therefore not directly comparable.

A closer comparison is provided by Gao and Zhao \cite{GaoZhao2021UnitaryHecke}, who studied primitive ray class characters over \(\mathbb Q(i)\) with conductor \(\mathfrak f=((1+i)^3q)\), where \(q\in\mathbb Z[i]\) and \((q,2)=1\). Their Corollaries~1.2 and~1.4 give a nonvanishing proportion \(\gg1/\log\norm(\mathfrak f)\) unconditionally and at least \(1/2+\littleo(1)\) under the generalized Riemann hypothesis. Our theorem improves their unconditional bound to \(1/3-\littleo(1)\) and extends the result to arbitrary imaginary quadratic fields and arbitrary integral moduli.

A little bit of class field theory shows that ray class characters are the natural extension of Dirichlet characters to any number field. Their \(L\)-functions also satisfy a functional equation (Lemma~\ref{thm:functional-equation}), so it is natural to expect \Cref{thm:main} to hold for all number fields \(K\). The difficulty in extending beyond \(K=\mathbb{Q}\) and \(K=\mathbb{Q}(\sqrt{-d})\) lies in part in the infinite unit groups of all other number fields. These make it possible for the number of ray class characters modulo \(\mathfrak{f} \) to be much smaller than \(\norm(\mathfrak{f}) \). See Lemmas \ref{lem:ray-class-exact-sequence} and \ref{lem:primitive-family-size}. It would be interesting to see if some restrictions can be put on general \(K\) and \(\mathfrak{f}\) to recover a nonzero proportion of nonvanishing. The estimates in Sections \ref{sec:lattice bounds} and \ref{sec:character sums} also use the imaginary quadratic assumption; it would be nice to also see these extended to arbitrary number fields.

We have already mentioned the application of the first two mollified moments to obtain \Cref{thm:main}. These moments might be of independent interest, so we conclude this introduction by recording them here. Let \(\mu\) be the M\"obius
function on integral ideals. Associated to a ray class character \(\chi\) modulo \(\mathfrak{f}\) is our mollifier
\begin{equation}\label{eq:mollifier}
\mathcal M(\chi)
=\sum_{\substack{\norm(\mathfrak m)\le M\\
(\mathfrak m,\mathfrak f)=1}}
\frac{\widetilde{\mu}(\mathfrak m)\chi(\mathfrak m)}
{\norm(\mathfrak m)^{1/2}}
\end{equation}
where
\[
	\widetilde{\mu}(\mathfrak{m}) = \mu(\mathfrak{m})\max \biggl( 1-\frac{\log\norm(\mathfrak m)}{\log M},0 \biggr).
\]

Throughout the paper we use \(\theta>0\) to parameterize the length of the mollifier as \( M = \norm(\mathfrak{f})^\theta \). The value of \(\theta\) is fixed as \( \norm(\mathfrak{f}) \) varies.

\begin{proposition} \label{prop:first mollified moment}
    Let \(K\) be an imaginary quadratic field. For \( 0 < \theta < 1/2 \) we have as \( \norm(\mathfrak{f}) \to \infty \) that
    \[
    \sideset{}{^*}\sum_{\chi\bmod\mathfrak f} \mathcal M(\chi) L\left(\frac{1}{2},\chi\right)
    =(1+\mathsf{o}_{K,\theta}(1)) h_K^*(\mathfrak{f}).
    \]
    The asterisk restricts the sum to primitive ray class characters modulo \(\mathfrak f\).
\end{proposition}

\begin{proposition} \label{prop:second mollified moment}
With the same setup as in Proposition~\ref{prop:first mollified moment}, we have
\[
    \sideset{}{^*}\sum_{\chi\bmod\mathfrak f}
    \left|\mathcal M(\chi)L\!\left(\frac12,\chi\right)\right|^2
    =\left(1+\frac1\theta+\littleo_{K,\theta}(1)\right)
    h_K^*(\mathfrak f).
\]
\end{proposition}

These formulas follow from the more precise estimates
\eqref{C asymptotic first moment} and \eqref{D asymptotic second moment},
together with Lemma~\ref{lem:primitive-family-size}.
Cauchy's inequality applied to the two moments gives a nonvanishing proportion of at least \(\theta/(1+\theta)+\littleo_{K,\theta}(1)\). For each fixed \(0<\theta<1/2\), first take the lower limit as \(\norm(\mathfrak f)\to\infty\) over nonempty primitive families. Taking the supremum of the resulting bounds over these fixed values of \(\theta\) then gives \(1/3\), as asserted in Theorem~\ref{thm:main}.

In Section~\ref{sec:ray-class-preliminaries} we recall the first properties of ray class characters
and their associated Hecke \(L\)-functions.
Section~\ref{sec:analytic-preliminaries} establishes various lemmas used throughout the rest of the paper. Sections~\ref{sec:first-moment} and
\ref{sec:second-moment} evaluate the twisted first and second moments, and Propositions \ref{prop:first mollified moment} and \ref{prop:second mollified moment} are proven in the subsequent Sections \ref{sec:mollified-first-moment} and \ref{sec:mollified-second-moment}. The proof of Theorem~\ref{thm:main} is detailed in Section~\ref{sec:main thm}.

Throughout this paper we maintain the notation \( K \supseteq \mathcal{O}_K \) for a number field and its ring of integers. The letter \(\mathfrak{f}\) denotes a nonzero integral ideal of \(\mathcal{O}_K \). All asymptotic estimates involving \(\mathfrak{f}\) hold as \(\norm(\mathfrak{f}) \) tends to infinity. The number-theoretic functions \(\phi(\mathfrak{f}) \), \(d(\mathfrak{f}) \), \( \mu(\mathfrak{f}) \) have their usual definitions. We also write
\[
	\eta(\mathfrak f) =\sum_{\mathfrak p\mid\mathfrak f} \frac{\log\norm(\mathfrak p)}{\norm(\mathfrak p)-1}.
\]
Implicit constants may depend on the field \(K\). At times we use the letter \(\varepsilon \) that indicates a small constant that may change from line to line, and implicit constants may depend on \(\varepsilon \) as well.

\section{Preliminaries}\label{sec:ray-class-preliminaries}
In this section we review the development of \( L \)-functions associated to characters of the ray class group. We first recall the Dedekind zeta function. The norm
of a nonzero integral ideal \(\mathfrak a\) of \(K\) is
\(\norm(\mathfrak a) \coloneqq [\mathcal O_K:\mathfrak a]\), and this definition is extended
multiplicatively to fractional ideals. The norm of a principal ideal is related to the norm of an element by \(\norm((\alpha)) = |\norm(\alpha)| \). We denote the first few terms of the Laurent expansion of the Dedekind zeta function at \(s=1\) by
\[
\zeta_K(s) =
\sum_{\mathfrak{a}} \frac{1}{\norm(\mathfrak{a})^s} =\frac{\gamma_{-1}(K)}{s-1}+\gamma_0(K)+\bigO_K(s-1).
\]
The coefficients here have been computed by Hashimoto, Iijima, Kurokawa, and Wakayama in \cite{HashimotoEtAl2004}.

We shall use a bound for \(1/\zeta_K(s)\) near \(\Re s=1\), in the proof of Lemma~\ref{lem:conrey-1}. For sufficiently small \(C>0\), the region \(\sigma\ge1-C/\log(2+|t|)\) contains at most one zero of \(\zeta_K\), as established by Stark in \cite{Stark1974}; see also \cite[Theorem~5.33]{Analyticbook}. Any such exceptional zero is real and lies to the left of \(1\). Since \(K\) is fixed, we may decrease \(C\) to exclude it. The standard argument for the reciprocal zeta function, as in \cite[Section~6.1]{MontgomeryVaughan2006}, then gives
\begin{equation}\label{eq:bound for Dedekind zeta in zero free region}
    \begin{aligned}
    \zeta_K(\sigma+\mathsf{i}t)^{-1} &\ll_K\log(2+|t|),
    \qquad \mathrlap{\sigma\ge1-\frac{C}{\log(2+|t|)},} \\
    \zeta_K(\sigma+\mathsf{i}t)^{-1} &\ll_K \log (2+ |t|).
    \end{aligned}
\end{equation}
At times we will deal with \( \zeta_K \) with certain factors removed from its Euler product. Accordingly, for an integral ideal \( \mathfrak{b} \) we set
\[
	\zeta_{K,\mathfrak{b}}(s) \coloneqq \sum_{(\mathfrak{a},\mathfrak{b}) = 1} \frac{1}{\norm(\mathfrak{a})^s} = \prod_{\mathfrak{p} \nmid \mathfrak{b}} \biggl(1 - \frac{1}{(\norm(\mathfrak{p}))^s}\biggr)^{-1}.
\]

\subsection{Ray class group}

Let \(I_K\) be the group of nonzero fractional ideals of \( K \), with
subgroup \(P_K\) of principal fractional ideals.  We use moduli with trivial infinite part, so no sign conditions are imposed at real embeddings.
For a nonzero integral ideal \(\mathfrak f\), put
\[I_K(\mathfrak f)
\coloneqq\{\mathfrak a\in I_K:(\mathfrak a,\mathfrak f)=1\} \quad \text{ and } P_K(\mathfrak f)
\coloneqq\{(\alpha)\in P_K:\alpha\equiv1\pmod{\mathfrak f}\}.\]
For fractional values \( \alpha \in K \), the
congruence \(\alpha\equiv1\pmod{\mathfrak f}\) means that
\(\alpha=a/b\) for some \(a,b\in\mathcal O_K\) where \((ab,\mathfrak f)=1\) and
\(a\equiv b\pmod{\mathfrak f}\) in the usual sense.
The ray class group modulo \(\mathfrak f\) is defined as \(\operatorname{Cl}_K(\mathfrak f) \coloneqq I_K(\mathfrak f)/P_K(\mathfrak f)\).

Write
\(h_K(\mathfrak f)\coloneqq|\operatorname{Cl}_K(\mathfrak f)|\) and
\(h_K\coloneqq|I_K/P_K|\). We also put
\(K^{\mathfrak f,1}=\{x\in K^\times:x\equiv1\pmod{\mathfrak f}\}\).

\begin{lemma}\label{lem:ray-class-exact-sequence}
There is an exact sequence
\[
    \begin{tikzcd}
    1 \ar[r] & \mathcal{O}_K^\times \cap K^{\mathfrak{f},1} \ar[r] & \mathcal{O}_K^\times \ar[r] & (\mathcal{O}_K/\mathfrak{f})^\times \ar[r] & \operatorname{Cl}_K(\mathfrak{f}) \ar[r] & I_K/P_K \ar[r] & 1.
    \end{tikzcd}
\]
Consequently,
\[
h_K(\mathfrak f)
=h_K\frac{\phi(\mathfrak f)}
{[\mathcal O_K^\times:\mathcal O_K^\times\cap K^{\mathfrak f,1}]}.
\]
\end{lemma}

\begin{proof}
See \cite[Section~3.2.1]{Cohen} or \cite[Theorem~V.1.7]{MilneCFT}, taking the infinite part of the modulus to be trivial.
\end{proof}

\subsection{Ray class characters}

For \(K=\mathbb Q\), characters of this finite-modulus ray class group correspond to even Dirichlet characters.

Here and throughout, we refer to characters of \( \operatorname{Cl}_K(\mathfrak{f}) \) simply as ``characters modulo \(\mathfrak{f}\).'' We extend the definition of a character modulo \( \mathfrak{f} \) to \( I_K(\mathfrak{f}) \) by pulling back along \( I_K(\mathfrak{f}) \to \operatorname{Cl}_K(\mathfrak{f}) \). We further extend the definition of such a \( \chi \) to all fractional ideals by setting \( \chi(\mathfrak{a}) = 0 \) for \( (\mathfrak{a},\mathfrak{f}) \neq 1 \).

Let \(\chi^* \) and \(\chi \) be characters of moduli \(\mathfrak{f}^* \) and \(\mathfrak{f} \) with
\( \mathfrak{f}^* \mid \mathfrak{f} \). If
\(\chi^*(\mathfrak{a})=\chi(\mathfrak{a})\) for every
\(\mathfrak{a}\in I_K(\mathfrak{f})\), then \(\chi\) is
\textit{induced} by \(\chi^* \). The conductor of \( \chi \) modulo
\(\mathfrak{f}\) is the smallest modulus \( \mathfrak{c} \) such that
\(\chi\) is induced modulo \( \mathfrak{c} \), and \( \chi \) is primitive
if \( \mathfrak{c} = \mathfrak{f} \).

We write \(\sum_{\chi \bmod\mathfrak f}\) and
\(\sideset{}{^*}\sum_{\chi\bmod\mathfrak f}\) for sums over ray class
characters modulo \( \mathfrak{f} \); the first sum
is over all such characters, and the starred sum restricts to primitive
characters. For every fractional ideal \(\mathfrak{m} \neq 0 \), we have the orthogonality relation
\[
\sum_{\chi \text{ mod } \mathfrak{f}} \chi(\m)
=
\begin{cases}
h_K(\f), & \text{if } \m \in P_K(\f),\\
0, & \text{otherwise}.
\end{cases}
\]
The analogous sum over primitive characters is computed in the following
\begin{lemma}\label{lem:primitive-ray-orthogonality}
For every nonzero fractional ideal \(\mathfrak m\) with \((\mathfrak m,\mathfrak f)=1\),
\[
\sideset{}{^*}\sum_{\chi \bmod \mathfrak f}\chi(\mathfrak m)
=
\ssum_{\substack{\mathfrak m\in P_K(\mathfrak t)\\
\mathfrak s\mathfrak t=\mathfrak f}}
\mu(\mathfrak s)h_K(\mathfrak t).
\]
In particular, the number of primitive characters modulo \( \mathfrak{f} \) is
\begin{equation} \label{eq:}
h_K^*(\mathfrak f)
=
\sum_{\mathfrak s\mathfrak t=\mathfrak f}
\mu(\mathfrak s)h_K(\mathfrak t).
\end{equation}
\end{lemma}

\begin{proof}
Every character modulo \(\mathfrak f\) is induced from a unique primitive
character whose conductor divides \(\mathfrak f\). Then the first result follows from the orthogonality relation and M\"obius inversion. The second result follows by taking \(\mathfrak m=\mathcal O_K\).
\end{proof}

\subsection[Hecke L-functions]{Hecke \(L\)-functions}

The Hecke \(L\)-function associated to a ray class character \(\chi\) is defined for  \( \Re s  > 1 \) by
\[
L(s,\chi)
:=
\sum_{\mathfrak{a}} \frac{\chi(\mathfrak a)}{\norm(\mathfrak a)^s},
\]
where the sum is over all nonzero integral ideals in \( \mathcal{O}_K \). It admits an Euler product over prime ideals \( \mathfrak{p} \),
\[
L(s,\chi)
=
\prod_{\mathfrak{p}}
\left(
1-\frac{\chi(\mathfrak p)}{\norm(\mathfrak p)^s}
\right)^{-1}.
\]
If \( \chi \) has modulus \( \mathfrak{f} \), then the product is equivalently written by restricting to primes \( \mathfrak{p} \nmid \mathfrak{f} \).

For the rest of this subsection, assume that \(K\) is imaginary quadratic. For a primitive ray class character \(\chi\) of conductor \(\mathfrak f\), the completed \(L\)-function is
\[
\varLambda(s,\chi)
:=
\bigl(|d_K|\norm(\mathfrak f)\bigr)^{s/2}
(2\pi)^{-s}\Gamma(s)L(s,\chi),
\]
where \(d_K\) is the discriminant of \(K\).
Let \(\mathfrak D_K\) denote the different of \(K/\mathbb Q\), and \(\tau(\chi)\) the Gauss sum associated with \(\chi\). Following \cite[Section~3.8, Exercise~11]{Analyticbook}, choose a nonzero integral ideal \(\mathfrak e\) coprime to \(\mathfrak f\) and \(\beta\in K^\times\) such that \((\beta)=\mathfrak e(\mathfrak D_K\mathfrak f)^{-1}\). Both may be chosen independently of \(\chi\), and \(\beta\in(\mathfrak D_K\mathfrak f)^{-1}\). Then

\[
\tau(\chi)
=
\chi(\mathfrak e)
\sum_{\substack{\alpha\bmod\mathfrak f\\
(\alpha,\mathfrak f)=(1)}}
\chi_{\mathfrak f}(\alpha)
\exp\!\left(
2\pi i\Tr(\alpha\beta)
\right),
\]
where
\(\chi_{\mathfrak f}(\alpha\bmod\mathfrak f):=\chi((\alpha))\) for
\(\alpha\in\mathcal O_K\) with \((\alpha,\mathfrak f)=1\).  This is the
restriction of \(\chi\) to principal ideals, viewed as a character of
\((\mathcal O_K/\mathfrak f)^\times\).  We also have
\(|\tau(\chi)|=\norm(\mathfrak f)^{1/2}\) by
\cite[Section~3.8, Exercise~12]{Analyticbook}.

\begin{lemma}[Functional equation]
\label{thm:functional-equation}
Let \(\chi\) be a nontrivial primitive ray class character of conductor
\(\mathfrak f \).
The completed \(L\)-function associated to \( \chi \) is entire and satisfies the functional equation
\[
\varLambda(s,\chi)
=
W(\chi)\,
\varLambda(1-s,\overline{\chi}),
\]
for a complex number \( W(\chi) \) of modulus \( 1\) (the root number).
With the above normalization,
\(W(\chi)=\tau(\chi)/\norm(\mathfrak f)^{1/2}\).

\end{lemma}

\begin{lemma}[Approximate functional equation I]\label{approximate functional equation}
Fix a weight function
\[
V(y)
=
\frac{1}{2\pi i}
\int_{(3)}
G(s)
\left(\frac{\sqrt{|d_K|}}{2\pi}\right)^s
\frac{\Gamma(\frac{1}{2}+s)}{\Gamma(\frac{1}{2})}
y^{-s}\frac{ds}{s}
\]
where \( G \) is an even entire function, decaying faster than any polynomial in vertical strips with \( G(0) = 1 \). For a nontrivial primitive character \(\chi\) of conductor \(\mathfrak f\),
\[
L\!\left(\frac12,\chi\right)
= \sum_{\substack{\mathfrak a\subseteq\mathcal O_K\\ \mathfrak a\ne(0)}}
\frac{\chi(\mathfrak a)}{\norm(\mathfrak a)^{1/2}}
V\!\left(
\frac{\norm(\mathfrak a)}
{\norm(\mathfrak f)^{1/2}}
\right) +
\frac{\tau(\chi)}{\norm(\mathfrak f)^{1/2}}
\sum_{\substack{\mathfrak a\subseteq\mathcal O_K\\
\mathfrak a\ne(0)}}
\frac{\overline{\chi}(\mathfrak a)}{\norm(\mathfrak a)^{1/2}}
V\!\left(
\frac{\norm(\mathfrak a)}
{\norm(\mathfrak f)^{1/2}}
\right).
\]
\end{lemma}

See \cite[Theorem~5.3]{Analyticbook} for details. For any fixed \(A>0\), we may choose \(G(s)=P_A(s)\mathsf e^{s^2}\), where \(P_A\) is an even polynomial with \(P_A(0)=1\) that vanishes at all poles of \(\Gamma(\tfrac12+s)\) in \(\Re s\ge-A\). Shifting the contour to \(\Re s=-A\) or \(\Re s=A\) then shows that \(V\) is an example of a function \(f\) on \([0,\infty)\), smooth on \( (0,\infty)\), satisfying the estiamtes
\begin{equation}\label{eq:approximation to indicator}
    f(y)=1+\mathsf O_A(y^A)\quad(y\to0^+),
    \qquad f^{(j)}(y)=\mathsf O_{A,j}(y^{-A})\quad(y\to\infty).
\end{equation}
If \(P_A\) is chosen with real coefficients, then \(G\) is real-valued on the real axis.

\begin{lemma}[Approximate functional equation II]\label{approximate functional equation II}
Let \(\chi\) be a nontrivial primitive character of conductor \(\mathfrak f\). Let \(G(s)\)
be an even entire function, decaying faster than any polynomial in every vertical strip, and
normalized by \(G(0)\Gamma(1/2)=1\). Define
\[
\lambda_\chi(n)
:=
\sum_{\substack{\mathfrak a,\mathfrak b\subseteq\mathcal O_K\\
\mathfrak a,\mathfrak b\ne(0)\\
\norm(\mathfrak a)\norm(\mathfrak b)=n}}
\chi(\mathfrak a)\overline{\chi}(\mathfrak b),
\]
so that
\(L(s,\chi)L(s,\overline\chi)=\sum_{n\ge1}\lambda_\chi(n)n^{-s}\)
for \(\Re s>1\). Then
   \[
   	\left|L\left(\frac{1}{2},\chi\right)\right|^2
	=
	2\sum_{n=1}^{\infty} \frac{\lambda_{\chi}(n)}{n^{1/2}}
	W\left(\frac{4n\pi^2}{|d_{K}|\norm(\f)}\right)
   \]
   where
   \[
   	W(y)
	=
	\frac{1}{2\pi i}\int_{(3)}
	\Gamma^2\left(s+\frac{1}{2}\right)G^2(s)s^{-1}y^{-s}\,\mathrm{d}s.
   \]
\end{lemma}
\begin{proof}
    Consider the integral
	\[
		I(\chi)
		=
		\frac{1}{2\pi i}
		\int_{(3)}
		\varLambda(\tfrac{1}{2} + s, \chi)
		\varLambda(\tfrac{1}{2} + s, \overline{\chi})
		G^2(s)\frac{\mathrm{d}s}{s}.
	\]
    Then by Cauchy's theorem, we have
    \[
    	\varLambda\left(\frac{1}{2},\chi\right)
	\varLambda\left(\frac{1}{2},\overline{\chi}\right)
	G^2(0)
	=
	2I(\chi).
    \]
Since
\[
\varLambda\left(\frac12,\chi\right)
\varLambda\left(\frac12,\overline\chi\right)
=
\frac{\sqrt{|d_K|\norm(\mathfrak f)}}{2\pi}
\Gamma(1/2)^2\left|L\left(\frac12,\chi\right)\right|^2,
\]
the stated normalization cancels the gamma factors. Expanding the
two \(L\)-functions in \(I(\chi)\) and grouping the ideals according to
\(n=\norm(\mathfrak a)\norm(\mathfrak b)\), we get
\[
	\left|L\left(\frac{1}{2},\chi\right)\right|^2
	=
	2\sum_{n=1}^{\infty}\frac{\lambda_\chi(n)}{n^{1/2}}
	\frac{1}{2\pi i}
	\int_{(3)}
	\Gamma^2\left(s+\frac{1}{2}\right)G^2(s)s^{-1}
	\left(\frac{|d_{K}|\norm(\f)}{4n\pi^2}\right)^s
	\,\mathrm{d}s.
\]
\end{proof}

By choosing the even polynomial factor in \(G\) to cancel the gamma poles in \(\Re s\ge-A\), the weight \(W\) can likewise be chosen to satisfy \eqref{eq:approximation to indicator}.

\section{Preliminary lemmas} \label{sec:analytic-preliminaries}
\subsection{Sums over primes}

We recall the standard ideal-counting estimate
from \cite{marcus1977number} and the number-field Mertens estimates of
\cite{Rosen1999Mertens}.
\begin{lemma}\label{lem:ideal-density-mertens}
As \(x\to\infty\),
\begin{equation} \label{eq:ideal-density}
\sum_{\norm(\mathfrak a)\le x}1
=
\gamma_{-1}(K)x
+\bigO_K\!\left(x^{1-1/[K:\mathbb Q]}\right),
\end{equation}
recalling the notation \(\gamma_{-1}(K)=\operatorname*{res}_{s=1}\zeta_K(s)\). Moreover,
\begin{align}
\prod_{\norm(\mathfrak p)\le x}
\left(1-\frac1{\norm(\mathfrak p)}\right)^{-1}
&\sim e^\gamma\gamma_{-1}(K)\log x, \label{eq:mertens product}\\
\sum_{\norm(\mathfrak p)\le x}\frac1{\norm(\mathfrak p)}
&\sim\log_2x, \label{eq:mertens reciprocal}\\
\sum_{\norm(\mathfrak p)\le x}
\frac{\log\norm(\mathfrak p)}{\norm(\mathfrak p)}
&=\log x+\bigO_K(1). \label{eq:mertens log over norm}
\end{align}
\end{lemma}

These asymptotics can be interpreted as the truncation of certain Dirichlet series at \( s=1 \). We will need the upper bounds to hold uniformly slightly through the left of \( \Re s=1 \). This is the content of the following

\begin{lemma} \label{lem:uniform Mertens}
	    Fix \( B, C > 0 \). As \( x \to \infty \), the following two
	    estimates hold:
	    \begin{align}
	        \prod_{\mathfrak{p} \mid \mathfrak{n}} (1 - \norm(\mathfrak{p})^{-\sigma})^{-1} &\ll_{K,B,C} \log_2 x, \label{eq:uniform mertens product} \\
        \sum_{\mathfrak p\mid\mathfrak n} \frac{\log\norm(\mathfrak p)}{\norm(\mathfrak p)^\sigma-1} &\ll_{K,B,C}\log_2 x. \label{eq:uniform mertens log over norm}
        \end{align}
		uniformly in \( \norm(\mathfrak{n}) \le x^B \) and \( \sigma \ge 1-1/(C \log_2 x)\).
\end{lemma}

\begin{proof}
Both expressions decrease as \(\sigma\) increases, so it suffices to take
\(\sigma=1-1/(C\log_2x)\). For sufficiently large \(x\), this value is
at least \(3/4\). Split the prime divisors of \(\mathfrak n\) at
\(\norm(\mathfrak p)=\log x\). Below this cutoff, Taylor's formula gives
\(\norm(\mathfrak p)^{-\sigma}
=\norm(\mathfrak p)^{-1}
(1+\bigO_C(\log\norm(\mathfrak p)/\log_2x))\).
Consequently, \eqref{eq:mertens log over norm} implies that
\(\sum_{\norm(\mathfrak p)\le\log x}
(\norm(\mathfrak p)^{-\sigma}-\norm(\mathfrak p)^{-1})\ll_{K,C}1\).
The terms of degree at least two in the logarithms of both Euler products
have bounded sum, since \(2\sigma\ge3/2\). Therefore
\[
\prod_{\norm(\mathfrak p)\le\log x}
(1-\norm(\mathfrak p)^{-\sigma})^{-1}
\ll_{K,C}
\prod_{\norm(\mathfrak p)\le\log x}
(1-\norm(\mathfrak p)^{-1})^{-1}
\ll_K\log_2x
\]
by \eqref{eq:mertens product}.
For the remaining prime divisors, the bound
\(\sum_{\mathfrak p\mid\mathfrak n}\log\norm(\mathfrak p)
\le\log\norm(\mathfrak n)\le B\log x\) gives
\[
\sum_{\substack{\mathfrak p\mid\mathfrak n\\\norm(\mathfrak p)>\log x}}
\norm(\mathfrak p)^{-\sigma}
\le\frac{B\log x}{(\log x)^\sigma\log_2x}
=\frac{Be^{1/C}}{\log_2x}.
\]
Its Euler product is bounded, which proves
\eqref{eq:uniform mertens product}. This argument also covers ideals of
bounded norm, including \(\mathfrak n=\mathcal O_K\), without a
separate assumption on \(\log_2\norm(\mathfrak n)\).

For \eqref{eq:uniform mertens log over norm}, the primes below the cutoff contribute
\(\ll_C\)
\(\sum_{\norm(\mathfrak p)\le\log x}
\log\norm(\mathfrak p)/\norm(\mathfrak p)\ll_K\log_2x\).
Above the cutoff it is
\(\ll_C(\log x)^{-\sigma}
\sum_{\mathfrak p\mid\mathfrak n}\log\norm(\mathfrak p)
\ll_{B,C}1\). This proves the second estimate.

\end{proof}

The next lemma develops an approximation for the number of primitive characters modulo \( \mathfrak{f} \) when \( \mathcal{O}_K \) has finitely many units. By Dirichlet's unit theorem, this occurs exactly when \( K = \mathbb{Q} \) or \( K \) is imaginary quadratic.

\begin{lemma}\label{lem:primitive-family-size}

Consider the Dirichlet convolution
\[
\Phi_K^*(\f)
:=(\mu*\phi)(\f)
=\sum_{\s\t=\f}\mu(\s)\phi(\t).
\]
We have
\[
\Phi_K^*(\f) =\norm(\f) \prod_{\p\parallel\f}\left(1-\frac{2}{\norm(\p)}\right) \prod_{\p^2\mid\f}\left(1-\frac{1}{\norm(\p)}\right)^2.
\]
In particular, if a (prime) ideal of norm \(2\) exactly divides \( \mathfrak{f} \), then \(h_K^*(\f)=0\).

Now assume \( \mathcal{O}_K^\times \) is finite. Then
\[
h_K^*(\f)=\frac{h_K}{\lvert\mathcal O_K^\times\rvert}\Phi_K^*(\f)+\bigO_K(1).
\]
Moreover, as \(\norm(\mathfrak f)\to\infty\) among ideals satisfying
\(\mathfrak p\parallel\mathfrak f\implies\norm(\mathfrak p)\ne2\),
\[
h_K^*(\f) \gg \frac{\norm(\f)} {\log_2^2\norm(\f)}.
\]
\end{lemma}

\begin{proof}
The product formula for \(\Phi_K^*(\mathfrak{f}) \) is readily checked on prime powers, and then it follows for all ideals by multiplicativity.  The cardinality
formula of Lemma~\ref{lem:ray-class-exact-sequence} gives
\[
    h_K(\mathfrak{t}) - \frac{h_K}{|\mathcal{O}_K^\times|} \phi(\mathfrak{t}) = h_K \phi(\mathfrak{t}) \frac{|\mathcal{O}_K^\times \cap K^{\mathfrak{t},1}| - 1}{|\mathcal{O}_K^\times|}.
\]
for any integral ideal \(\mathfrak{t} \).
The right-hand side is supported on \(\mathfrak{t}\) dividing \( \prod_{1 \neq u \in \mathcal{O}_K^\times} (u-1) \), and the count of such ideals depends only on \( K \). So the convolution of the right-hand side by the M\"obius function is trivially \( \mathsf{O}_K(1) \). By Lemma~\ref{lem:primitive-ray-orthogonality}, the convolution of the
left-hand side is exactly the difference in the asserted asymptotic.

Now suppose that \( \mathfrak{p} \) with norm \( 2 \) exactly divides \( \mathfrak{f} \) and write  \( \mathfrak{f} = \mathfrak{pg} \). Apply the exact sequence in
Lemma~\ref{lem:ray-class-exact-sequence} to \(\mathfrak{f} \) and
\(\mathfrak{g} \), and connect the sequences with vertical arrows:
\[
    \begin{tikzcd}
    1 \ar[r] & (\mathcal{O}_K/\mathfrak{f})^\times/\operatorname{im} \mathcal{O}_K^\times \ar[r] \ar[d] & \operatorname{Cl}_K(\mathfrak{f}) \ar[r]\ar[d] & I_K/P_K \ar[r] \ar[d] & 1 \\
    1 \ar[r] & (\mathcal{O}_K/\mathfrak{g})^\times/\operatorname{im} \mathcal{O}_K^\times \ar[r] & \operatorname{Cl}_K(\mathfrak{g}) \ar[r] & I_K/P_K \ar[r] & 1.
    \end{tikzcd}
\]
The left and right vertical arrows are isomorphisms - the left by the Chinese remainder theorem - so the middle one is also, per the five lemma. Thus every character modulo \( \mathfrak{f} \) is induced modulo \( \mathfrak{g} \), i.e.\ there are no primitive characters modulo \(\mathfrak{f} \).

Finally suppose no prime ideal of norm \(2\) exactly divides \(\mathfrak f\). We write
\[
	\Phi_K^*(\mathfrak{f}) = \norm(\mathfrak{f}) \prod_{\mathfrak{p} \mid \mathfrak{f}} \biggl(1 - \frac{1}{\norm(\mathfrak{p})}\biggr)^2 \prod_{\mathfrak{p} \Vert \mathfrak{f}} \frac{1-2/ \norm(\mathfrak{p})}{(1-1/ \norm(\mathfrak{p}))^2}
\]
Applying the elementary inequality
\(\frac{1-2/x}{(1-1/x)^2}
=1-\frac1{(x-1)^2}\ge1-\frac3{x^2}\) for \(x\ge3\)
to the product over \(\mathfrak p\parallel\mathfrak f\) shows that it is
\(\gg_K1\), since these factors are positive and
\(\sum_{\mathfrak p}\norm(\mathfrak p)^{-2}\) converges.
The asserted lower bound now follows from the prime ideal
theorem and the first estimate of Lemma~\ref{lem:uniform Mertens}.

\end{proof}

\subsection{Mollifier estimates}

The next lemma is an asymptotic formula for the derivatives of the mollification of \( 1/\zeta_{K,\mathfrak{af}}(s) \) near \( s=1 \). We show that the mollification saves a logarithmic factor in the error term on average. This averaging will occur in the estimates of Section~\ref{sec:mollified-second-moment}. This result is a generalization of \cite[Lemma 10]{Conrey1983ZerosLine}, which deals with the case \( K = \mathbb{Q} \). See also the discussion after the proof of \cite[Lemma 1]{MichelVanderKam2000}.

\begin{lemma}\label{lem:conrey-1}
	Write \( (x)_+ = \max(x,0) \). Fix a polynomial \(h\) satisfying \( h(0) = 0 \) and consider
	\[
		\frac{1}{\zeta_{K,\mathfrak{af}}}(s;h) \coloneqq \sum_{(\mathfrak{m,af}) = 1} \frac{\imu{m}}{\inorm[s]{m}} h \biggl( \biggl(\frac{ \log (M / \inorm{am})}{\log M}\biggr)_+ \biggr).
	\]
    We have the following asymptotic formula as
	\(\norm(\mathfrak{f}) \to \infty \), where the error terms are uniform in \( |s-1| \ll 1/\log M\) and in nonzero integral ideals \(\mathfrak a\) with \(\norm(\mathfrak a)\le M\), for each fixed integer \(k\ge0\).
	\[
		\biggl(\frac{1}{\zeta_{K,\mathfrak{a}\mathfrak{f}}}\biggr)^{(k)}(s;h) = \frac{A_k}{\prod_{\mathfrak{p} \mid \mathfrak{a}\mathfrak{f}} (1 - \norm(\mathfrak{p})^{-s})}  + \bigO_{K,\theta,h,k}(\log^k(M)\mathcal E_1(\mathfrak a))
	\]
	Here
	\[
		\begin{aligned}
			A_0 &= \frac{1}{\gamma_{-1}(K)}\left\{(s-1) h \biggl(\frac{\log(M/\norm(\mathfrak{a}))}{\log(M)}\biggr) + \frac{1}{\log(M)} h' \biggl(\frac{\log(M/\norm(\mathfrak{a}))}{\log(M)}\biggr)\right\}, \\
			A_1 &= \frac{1}{\gamma_{-1}(K)} h \biggl(\frac{\log(M/\norm(\mathfrak{a}))}{\log(M)}\biggr), \\
			A_k &= 0 \quad (k \ge 2), \\
			\mathcal{E}_1(\mathfrak{a})
				&= \frac{\log_2^2(M)}{\log^2(M)} [1 + \log (M) \log_2(M) (\norm(\mathfrak{a}) / M)^{b}], %
		\end{aligned}
	\]
	where \( b = 1/(C\log_2(M)) \) for sufficiently large
	\(C=C(K,\theta)>0\).
\end{lemma}

\textbf{Proof.} We follow the proof of \cite[Lemma 10]{Conrey1983ZerosLine}. Expanding \( h \) in Taylor series and applying the Mellin transform
\[
	\biggl(\frac{\log (M/\norm(\mathfrak{a} \mathfrak{m}))}{\log M}\biggr)_+^\ell = \frac{\ell!}{\log^\ell(M)} \times \frac{1}{2\pi \mathsf{i}} \int_{(3)} \frac{(M/\norm(\mathfrak{a} \mathfrak{m}))^{z}}{z^{\ell+1}} \,\mathrm{d}z
\]
gives
\[
	\frac{1}{\zeta_{K,\mathfrak{a}\mathfrak{f}}}(s;h)  = \sum_{\ell \ge 1} \frac{h^{(\ell)}(0)}{\log^\ell(M)} \times \frac{1}{2\pi \mathsf{i}} \sum_{(\mathfrak{m}, \mathfrak{a} \mathfrak{f}) = 1} \int_{(3)} \frac{\mu(\mathfrak{m})}{\norm(\mathfrak{m})^{z+s}} \frac{(M/\norm(\mathfrak{a}))^z}{z^{\ell + 1}} \,\mathrm{d}z.
\]
We interchange the inner sum and integral, change variables \( z \leadsto z-s \) and straighten the contour to get
\begin{equation} \label{reciprocal zeta integral kernel formula}
	\frac{1}{\zeta_{K,\mathfrak{a} \mathfrak{f}}}(s;h) = \sum_{\ell \ge 1} \frac{h^{(\ell)}(0)}{\log^\ell(M)} \times \frac{1}{2 \pi \mathsf{i}} \int_{(2)} \frac{1}{\zeta_{K,\mathfrak{a} \mathfrak{f}}(z)} \frac{(M/\norm(\mathfrak{a}))^{z-s}}{(z-s)^{\ell + 1}} \,\mathrm{d}z
\end{equation}
To evaluate the integral in \eqref{reciprocal zeta integral kernel formula} we deform the contour to the line \(\Re z=1\) with a left indentation near \(z=1\), as in the diagram below. Let \(I(s)\) denote the sum of the integrals over the three
segments \(\Re z=1\), \(\Im z=\pm\log M\), and \(\Re z=1-b\), so that the original integral over \(\Re z = 2 \) is equal to
\(-I(s)\) plus \(2\pi\mathsf i\) times the residue of the integrand at \(z=s\), with the orientations shown.

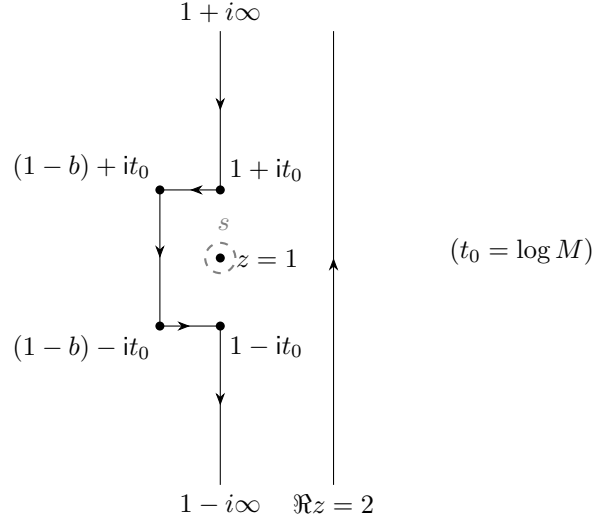
\begin{center}
	\begin{tikzpicture}[
	baseline=0pt,
    >=Stealth,
    arrow/.style={
        decoration={markings, mark=at position #1 with {\arrow{>}}},
        postaction={decorate}
    },
    arrow/.default=0.5,
    pt/.style={circle, fill=black, inner sep=1.2pt}
]

    \def\vpos{2.5} %
    \def\vlength{3.0} %

    \def\bone{1.0} %
    \def\bdist{0.8} %
    \def\bm {\bone-\bdist} %
    \def\tone{0.9}  %
    \def\tmax{3.0} %

    \coordinate (A) at (\bone, \tmax);    %
    \coordinate (B) at (\bone, \tone);    %
    \coordinate (C) at (\bm, \tone);      %
    \coordinate (D) at (\bm, -\tone);     %
    \coordinate (E) at (\bone, -\tone);   %
    \coordinate (F) at (\bone, -\tmax);   %

    \coordinate (Lbot) at (\vpos, -\vlength);
    \coordinate (Ltop) at (\vpos, \vlength);

    \coordinate (ZOne) at (\bone, 0);

    \draw[arrow=0.5] (Lbot) -- (Ltop);

    \draw[arrow=0.5] (A) -- (B);

    \draw[arrow=0.5] (B) -- (C);

    \draw[arrow=0.5] (C) -- (D);

    \draw[arrow=0.5] (D) -- (E);

    \draw[arrow=0.5] (E) -- (F);

    \node[below] at (\vpos, -\vlength) {$\Re z = 2$};

    \node[pt] at (ZOne) {};
    \node[right=2.4pt] at (ZOne) {$z=1$};

    \draw[gray, thick, dashed] (ZOne) circle (0.2cm);
    \node[gray, above left] at (\bone+0.25, 0.25) {$s$};

    \newcommand{\pointlabel}[3]{%
        \node[pt] at (#1) {};
        \node[#2] at (#1) {#3};
    }

    \pointlabel{B}{above right}{$1 + \mathsf{i}t_0$}
    \pointlabel{C}{above left}{$(1-b) + \mathsf{i}t_0$}
    \pointlabel{D}{below left}{$(1-b) - \mathsf{i}t_0$}
    \pointlabel{E}{below right}{$1 - \mathsf{i}t_0$}

    \node[above] at (A) {$1 + i\infty$};
    \node[below] at (F) {$1 - i\infty$};

\end{tikzpicture}

	\( \mathrlap{\qquad (t_0 = \log M)} \)
\end{center}
On the new contour we recall the bounds for \(\zeta_K\) in \eqref{eq:bound for Dedekind zeta in zero free region}, as well as the Mertens estimate \eqref{eq:uniform mertens product} of Lemma~\ref{lem:uniform Mertens}:
\[
    \begin{aligned}
		\zeta_K(\sigma + \mathsf{i}t)^{-1} &\ll_K \log_2(M) \quad \text{ for } \sigma \ge 1-b, |t| \le \log M; \\
		 \zeta_K(1+\mathsf{i}t)^{-1}
										   &\ll_K\log(2+|t|),
		 \quad \text{for } |t|\ge 1;  \\
        \prod_{\mathfrak{p} \mid \mathfrak{a}\mathfrak{f}}
		(1-\norm(\mathfrak{p})^{-z})^{-1}
										   &\ll_{K,\theta} \log_2(M)
        \quad \text{ for } \Re z \ge 1-b.
    \end{aligned}
\]
The finite Euler product estimate applies because \(\norm(\mathfrak a\mathfrak f)\le M^{1+1/\theta}\). These bounds control \(\zeta_{K,\mathfrak{af}}(z)^{-1}\) in the integrand defining \(I(s)\). The other factor is a function of \(s\). We estimate its \( s \)-derivatives using Cauchy's theorem, integrating around a circle about \( s = 1 \). We find that the optimal radius is about \( 1/\log M \), giving

	\[
		\partial_s^k \biggl[ \frac{(M/\norm(\mathfrak{a}))^{z-s}}{(z-s)^{\ell+1}}\biggr]  \ll_{k,\ell} \frac{\log^k(M) (M/\norm(\mathfrak{a}))^{\Re z - 1}}{|z-1|^{\ell+1}}
	\]
	uniformly for \( |s-1| \ll 1/\log M \).

    Armed with the above bounds, we can estimate \( I(s) \). Since \( \ell \ge 1 \), the integrals along \( \Re z = 1 \) can be differentiated in \( s \) under the integral sign. Doing so gives
	\[
		\partial^k_s \int_{\Re z = 1} \ll \log^k(M) \log_2(M) \int_{t_0}^{\infty} \frac{\log(3+t)}{t^{\ell + 1}} \,\mathrm{d}t \ll \log^{k-\ell}(M) \log_2^2(M)
	\]
	The horizontal integrals satisfy
	\[
		\partial_s^k \int_{\Im z = \pm t_0} \ll \log^{k-(\ell+1)}(M) \log_2^2(M) \times \int_{1-b}^1 (M/\norm(\mathfrak{a}))^{\sigma - 1} \,\mathrm{d}\sigma \ll \log^{k-(\ell+1)}(M)\log_2(M)
	\]
	Finally, the last integral satisfies
	\[
        \begin{aligned}
		\partial_s^k \int_{\Re z = 1-b}
        &\ll \log^k(M) \log_2^2(M) (\norm(\mathfrak{a})/M)^{b} \int_{-t_0}^{t_0} \frac{\mathrm{d}t}{(b^2 + t^2)^{(\ell + 1)/2}} \\
        &\ll  \log^k(M) \log_2^{\ell+2}(M) (\norm(\mathfrak{a})/M)^b.
        \end{aligned}
	\]
    We finally sum over \(\ell\) to obtain the error term in \eqref{reciprocal zeta integral kernel formula}.
    \[
        \sum_{\ell \ge 1} \frac{h^{(\ell)}(0)}{\log^{\ell}(M)} \times \partial_s^k \frac{1}{2\pi\mathsf{i}} I(s) \ll \log^k(M) \mathcal{E}_1(\mathfrak{a})
    \]
    It remains to evaluate the main term, namely the residue
    \[
		\sum_{\ell \ge 1} \frac{h^{(\ell)}(0)}{\log^\ell(M)} \times \partial_s^k \operatorname*{res}_{z=s} \biggl( \frac{1}{\zeta_{K,\mathfrak{a f}}(z)} \frac{(M/\norm(\mathfrak{a}))^{z-s}}{(z-s)^{\ell+1}} \,\mathrm{d}z \biggr).
	\]
	The \( k \)th \(s\)-derivative of the residue is
	\[
		\begin{aligned}
			\partial_s^k \operatorname*{res}_{z=s}
			&= \frac{1}{\ell!} \partial_s^k \partial_z^\ell \biggl[\frac{(M/\norm(\mathfrak{a}))^{z-s}}{\zeta_{K,\mathfrak{a}\mathfrak{f}
			}(z)}\biggr]_{z=s} \\
			&= \frac{1}{\ell!} \sum_{q=0}^\ell \binom{\ell}{q} \log^{\ell-q} (M/\norm(\mathfrak{a})) \times (1/\zeta_{K,\mathfrak{a}\mathfrak{f}})^{(q+k)}(s).
		\end{aligned}
	\]
	We have
	\[
		\frac{1}{\zeta_{K,\mathfrak{a}\mathfrak{f}}(s)} = \frac{(s-1)/\gamma_{-1}(K) + \bigO_K((s-1)^2)}{\prod_{\mathfrak{p} \mid \mathfrak{a}\mathfrak{f}} (1 - \norm(\mathfrak{p})^{-s})} = \frac{s-1}{\gamma_{-1}(K)\prod_{\mathfrak{p} \mid \mathfrak{a}\mathfrak{f}} (1-\norm(\mathfrak{p})^{-s})} + \bigO_{K,\theta}(\log^{-2}(M) \log_2(M))
	\]
	The logarithmic derivative of the finite Euler product satisfies
	\[
		\sum_{\mathfrak p\mid\mathfrak a\mathfrak f}
		\left|\frac{\log\norm(\mathfrak p)}
		{\norm(\mathfrak p)^s-1}\right|
		\ll \log_2(M)
	\]
	by the second estimate of Lemma~\ref{lem:uniform Mertens}, since
	\(|s-1|\ll1/\log M=\littleo(b)\). We also have
	\( \zeta_K'/\zeta_K (s) = -(s-1)^{-1} + \bigO_K(1) \). Thus
    \[
	\begin{aligned}
		\biggl(\frac{1}{\zeta_{K,\mathfrak{a}\mathfrak{f}}}\biggr)'(s)
		&= -\biggl[
		\sum_{\mathfrak p\mid\mathfrak a\mathfrak f}
		\frac{\log\norm(\mathfrak p)}{\norm(\mathfrak p)^s-1}
		+\frac{\zeta_K'}{\zeta_K}(s)
		\biggr]\frac{1}{\zeta_{K,\mathfrak a\mathfrak f}}(s)\\
		&= [(s-1)^{-1} + \bigO(\log_2(M))] \times \frac{(s-1)/\gamma_{-1}(K) + \bigO((s-1)^2)}{\prod_{\mathfrak{p \mid af}} (1 - \inorm[-s]{p})} \\
		&= \frac{1}{\gamma_{-1}(K) \prod_{\mathfrak{p} \mid \mathfrak{a}\mathfrak{f}} (1-\norm(\mathfrak{p})^{-s})} + \bigO(\log^{-1}(M) \log_2^2(M)).
	\end{aligned}
\]
For a general number of derivatives, by Cauchy's theorem applied to a circle of radius \( \asymp 1 / \log_2 (M) \) about \( s=1 \),
\[
	(1/\zeta_{K,\mathfrak{a}\mathfrak{f}})^{(q+k)}(s) \ll_{K,\theta,q,k} \log_2^{q+k}(M)
\]
These estimates imply that the error term in this residue sum is
\[
\ll_{K,\theta,h,k}\log^{k-2}(M)\log_2^2(M),
\]
which can be absorbed into \(\log^k(M)\mathcal E_1(\mathfrak a)\). For \(k=0\), the
main term from the sum over \(q\) is
\[
	\frac{1}{\gamma_{-1}(K)\prod_{\mathfrak{p} \mid \mathfrak{a}\mathfrak{f}} (1-\norm(\mathfrak{p})^{-s})} \times \biggl((s-1) \frac{\log^\ell (M/\norm(\mathfrak{a}))}{\ell!} + \frac{\log^{\ell-1} (M /\norm(\mathfrak{a}))}{(\ell-1)!}\biggr) %
\]
Summing over \( \ell \) yields \( A_0 \). The evaluations for \(k=1\) and \(k \ge 2\) are similar. \( \hfill \square \)

\begin{numberedcorollary}\label{lem:T0-T1-evaluation} For integral ideals \(\mathfrak l\) and \(\mathfrak d\), define
    \[
    \begin{aligned}
        T_0(\mathfrak{l})
        &= \sum_{(\mathfrak{m}, \mathfrak{lf}) = 1} \frac{\widetilde{\mu}(\mathfrak{lm})}{\norm(\mathfrak{m})}, \\
        T_1(\mathfrak l,\mathfrak d)
        &=\sum_{(\mathfrak{m,lf})=1}
        \frac{\mmu{lm}}{\inorm{m}}
        \log\frac{\inorm{f}}{\inorm{dm}}.
    \end{aligned}
    \]
    Uniformly in \( \mathfrak{l} \) coprime to \(\mathfrak f\) with \(\norm(\mathfrak{l}) \le M \) and
    \(\mathfrak d \) dividing \( \mathfrak l\),  we have
		\begin{align*}
			T_0(\mathfrak{l})
			&= \frac{1}{\gamma_{-1}(K) }\frac{\norm(\mathfrak{f})}{\phi( \mathfrak{f})} \frac{\mu_1(\mathfrak{l})}{\log(M)} + \bigO(\mathcal{E}_1(\mathfrak{l})), \\
			T_1(\mathfrak{l}, \mathfrak{d})
			&= \frac{1}{\gamma_{-1}(K)}\frac{\norm(\mathfrak{f})}{\phi( \mathfrak{f})}  \frac{\mu_1(\mathfrak{l})}{\log(M)} \log \frac{M \norm(\mathfrak{f})}{\inorm{ld}} + \bigO(\log(M) \mathcal{E}_1(\mathfrak{l})),
		\end{align*}
	where \( \mathcal{E}_1(\mathfrak{l}) \) is the error from Lemma~\ref{lem:conrey-1}, and
    \[
        \mu_1(\mathfrak{l}) = \frac{\mu(\mathfrak{l})}{\prod_{\mathfrak{p} \mid \mathfrak{l}} (1 - \norm(\mathfrak{p})^{-1})}.
    \]
\end{numberedcorollary}

\begin{proof}
	For \((\mathfrak m,\mathfrak l)=1\), the identity \(\widetilde\mu(\mathfrak l\mathfrak m)=\mu(\mathfrak l)\mu(\mathfrak m)\log(M/\norm(\mathfrak l\mathfrak m))/\log M\) holds on the mollifier support. Thus the estimate for \(T_0(\mathfrak l)\) follows from Lemma~\ref{lem:conrey-1} for \(k=0\) with \( h(x) = x \), \(s = 1\), and \( \mathfrak{a} = \mathfrak{l} \). For \( T_1(\mathfrak{l},\mathfrak{d}) \), we write \( \log \frac{\norm(\mathfrak{f})}{\norm(\mathfrak{d}\mathfrak{m})} = \log \frac{\norm(\mathfrak{f})}{\norm(\mathfrak{d})} - \log \norm(\mathfrak{m}) \) and apply Lemma~\ref{lem:conrey-1} with \( k=0,1 \). Since \(\norm(\mathfrak d)\le M\) and \( M = \norm(\mathfrak{f})^\theta \), the first logarithm on the right hand side satisfies
    \( \log (\norm(\mathfrak{f})/\norm(\mathfrak{d})) \ll_\theta \log M\).
    This yields the error term for \( T_1(\mathfrak{l},\mathfrak{d}) \).
\end{proof}

\subsection{Mean values at \texorpdfstring{\( s=1 \)}{s=1}}

\begin{lemma}\label{lem:reciprocal coprime mean value}
For any number field \(K\), nonzero integral ideal \(\mathfrak f\), and \(x\ge1\),
\[
\begin{aligned}
\ssum_{\norm(\mathfrak a)\le x\\(\mathfrak a,\mathfrak f)=1}
\frac1{\norm(\mathfrak a)}
&=\frac{\phi(\mathfrak f)}{\norm(\mathfrak f)}
\left\{\gamma_{-1}(K)\log x+\gamma_0(K)
+\gamma_{-1}(K)\eta(\mathfrak f)\right\}\\
&\quad+\bigO_K\!\left(\mathcal E_2(K,\mathfrak f)
 x^{-1/[K:\mathbb Q]}\right),
\end{aligned}
\]
where
\begin{equation}\label{eq:definition of E2}
\mathcal E_2(K,\mathfrak f)
=\prod_{\mathfrak p\mid\mathfrak f}
\left(1+\norm(\mathfrak p)^{-1+1/[K:\mathbb Q]}\right).
\end{equation}
If \([K:\mathbb Q]\ge2\), then
\(\log\mathcal E_2(K,\mathfrak f)
\ll_K \log^{[1/[K:\mathbb{Q}]}\norm(\mathfrak f)/
\log_2 \norm(\mathfrak f) \).
For \(K=\mathbb Q\), we have \(\mathcal E_2(\mathbb Q,f)=2^{\omega(f)}\),
where \(\omega(f)\) is the number of distinct prime divisors of \(f\).
\end{lemma}

\begin{proof}
The ideal density asymptotic \eqref{eq:ideal-density} reads
\(\sum_{\norm(\mathfrak a)\le x}1=\gamma_{-1}(K)x+R_K(x)\), where
\(R_K(x)=\bigO_K(x^{1-1/[K:\mathbb Q]})\). Partial summation gives
\(\zeta_K(s)-\gamma_{-1}(K)/(s-1)
=\gamma_{-1}(K)+s\int_1^\infty R_K(t)t^{-s-1}\,\mathrm dt\).
Letting \(s\to1\), we obtain
\begin{equation}\label{eq:Stieltjes coefficients identity}
\gamma_0(K)=\gamma_{-1}(K)+\int_1^\infty\frac{R_K(t)}{t^2}\,\mathrm dt.
\end{equation}
A second partial summation therefore yields
\begin{equation}\label{eq:reciprocal sum}
\sum_{\norm(\mathfrak a)\le x}\frac1{\norm(\mathfrak a)}
=\gamma_{-1}(K)\log x+\gamma_0(K)
+\bigO_K(x^{-1/[K:\mathbb Q]}).
\end{equation}
This estimate also holds for \(0<x<1\), since the sum is empty and
\(1+|\log x|\ll_K x^{-1/[K:\mathbb Q]}\).
The sum may be restricted to ideals coprime to \( \mathfrak{f} \) via the Möbius identity,
	\[
		\begin{aligned}
        \ssum_{\norm(\mathfrak{a}) \le x \\ (\mathfrak{a}, \mathfrak{f}) = 1} \frac{1}{\norm(\mathfrak{a})}
		&= \ssum_{\norm(\mathfrak{a}) \le x} \frac{1}{\norm(\mathfrak{a})} \sum_{\mathfrak{d} \mid (\mathfrak{a}, \mathfrak{f})} \mu(\mathfrak{d}) \\
		&= \sum_{\mathfrak{d} \mid \mathfrak{f}} \frac{\mu(\mathfrak{d})}{\norm(\mathfrak{d})} \sum_{\norm(\mathfrak{a}) \le x/\norm(\mathfrak{d})} \frac{1}{\norm(\mathfrak{a})}.
		\end{aligned}
	\]
	Applying \eqref{eq:reciprocal sum} at \(x/\norm(\mathfrak d)>0\) for every
\(\mathfrak d\mid\mathfrak f\) gives
    \begin{align*}
		\ssum_{\norm(\mathfrak{a}) \le x \\ (\mathfrak{a}, \mathfrak{f}) = 1} \frac{1}{\norm(\mathfrak{a})}
        &= (\gamma_{-1}(K) \log x + \gamma_0(K))
        \sum_{\mathfrak{d} \mid \mathfrak{f}} \frac{\mu(\mathfrak{d})}{\norm(\mathfrak{d})}
        - \gamma_{-1}(K)\sum_{\mathfrak{d} \mid \mathfrak{f}}
        \frac{\mu(\mathfrak{d})}{\norm(\mathfrak{d})} \log \norm(\mathfrak{d}) \\
        &\quad + \bigO_K\!\biggl(x^{-1/[K:\mathbb{Q}]}
        \sum_{\mathfrak{d} \mid \mathfrak{f}}
        \frac{\mu^2(\mathfrak{d})}{\norm(\mathfrak{d})^{1-1/[K:\mathbb{Q}]}}\biggr).
    \end{align*}
To evaluate the main terms, differentiate
\(\sum_{\mathfrak d\mid\mathfrak f}\mu(\mathfrak d)\norm(\mathfrak d)^{-s}
=\prod_{\mathfrak p\mid\mathfrak f}(1-\norm(\mathfrak p)^{-s})\)
and set \(s=1\). This gives respectively
\(\phi(\mathfrak f)/\norm(\mathfrak f)\) and
\(-\phi(\mathfrak f)\eta(\mathfrak f)/\norm(\mathfrak f)\), proving the main term. The divisor sum in the error term equals \(\mathcal E_2(K,\mathfrak f)\) by its Euler product. The bound for \(K \neq \mathbb{Q} \) follows from an application of the prime ideal theorem as in the proof of Lemma~\ref{lem:uniform Mertens}, while for \(K=\mathbb Q\) follows directly from that product.
\end{proof}

\begin{lemma}\label{lem:dirichlet-factorization}
	Let \( \mathfrak{f} \) be an integral ideal of \( K \). Let \( g \) be a multiplicative function on integral ideals, supported on squarefree ideals coprime to \( \mathfrak{f} \). Suppose that there is \(\delta > 0\) such that \(|g(\mathfrak p)-1| \ll \norm(\mathfrak p)^{-\delta}\), the implicit constant being uniform over prime ideals \(\mathfrak p\nmid\mathfrak f\).

	The function
	\[
		H(s;g) \coloneqq \zeta_{K,\mathfrak{f}}(s)^{-1} \prod_{\mathfrak{p} \nmid \mathfrak{f}} \biggl( 1 + \frac{g(\mathfrak{p})}{\norm(\mathfrak{p})^s}\biggr),
	\]
	initially defined on \( \Re s > 1 \), has Dirichlet series \[ H(s;g) = \sum_\mathfrak{m} h(\mathfrak{m})/\norm(\mathfrak{m})^s \] which converges absolutely for \( \Re s > \max(1-\delta,1/2) \).
    Moreover, \( g \) is equal to the Dirichlet convolution
    \[
    g(\mathfrak{l}) = \ssum_{\mathfrak{am} = \mathfrak{l} \\ (\mathfrak{a}, \mathfrak{f}) = 1}h(\mathfrak{m}). \]
\end{lemma}
\begin{proof} The local factor in the Euler product of \( H(s;g) \) expands as
	\[
		\left(1 + \frac{g(\mathfrak{p})}{\norm(\mathfrak{p})^s}\right)\!\left(1 - \frac{1}{\norm(\mathfrak{p})^s}\right) = 1 + \frac{g(\mathfrak{p}) - 1}{\norm(\mathfrak{p})^s} - \frac{g(\mathfrak{p})}{\norm(\mathfrak{p})^{2s}}.
	\]
	The deviation of this local factor
	from \( 1 \) is \( \ll \norm(\mathfrak{p})^{-\Re s -\delta} + \norm(\mathfrak{p})^{-2\Re s} \) , giving the domain of absolute convergence \( \Re s >
	\max(1-\delta, 1/2) \) as claimed. The formula for \( g \) is immediate.

\end{proof}

\begin{lemma}\label{lem:weighted reciprocal mean value}
Let \(g\) satisfy the hypotheses of Lemma~\ref{lem:dirichlet-factorization}
for some \(\delta>0\). For every fixed
\(0<\kappa<\min(1/[K:\mathbb Q],\delta,1/2)\) and \(x\ge1\),
\[
\sum_{\norm(\mathfrak l)\le x}\frac{g(\mathfrak l)}{\norm(\mathfrak l)}
=C_{K,\mathfrak f,\mathfrak g}\log x+D_{K,\mathfrak f,\mathfrak g}
+\bigO\!\left(\mathcal E_2(K,\mathfrak f)x^{-\kappa}
\log(\norm(\mathfrak f)x)\right),
\]
where
\[
\begin{aligned}
C_{K,\mathfrak f,\mathfrak g}
&=\frac{\phi(\mathfrak f)}{\norm(\mathfrak f)}\gamma_{-1}(K)H(1;g),\\
D_{K,\mathfrak f,\mathfrak g}
&=\frac{\phi(\mathfrak f)}{\norm(\mathfrak f)}
\left[\gamma_{-1}(K)H'(1;g)
+(\gamma_0(K)+\gamma_{-1}(K)\eta(\mathfrak f))H(1;g)\right].
\end{aligned}
\]
Here \(\mathcal E_2\) is defined in \eqref{eq:definition of E2}.
The implied constant depends only on \(K,\delta,\kappa\) and the uniform
constant in \(|g(\mathfrak p)-1|\ll\norm(\mathfrak p)^{-\delta}\).
\end{lemma}

\begin{proof}
We write \( g(\mathfrak{l}) = \sum_{\mathfrak{am} = \mathfrak{l}, (\mathfrak{a}, \mathfrak{f}) = 1} h(\mathfrak{m}) \) and apply Lemma~\ref{lem:reciprocal coprime mean value} to get
\begin{equation}\label{eq:weighted reciprocal coprime mean value expansion}
\begin{aligned}
\sum_{\norm(\mathfrak l)\le x}\frac{g(\mathfrak l)}{\norm(\mathfrak l)}
&= \sum_{\mathfrak{m}}  \frac{h(\mathfrak{m})}{\norm(\mathfrak{m})} \ssum_{\norm(\mathfrak a)\le x/\norm(\mathfrak m)\\ (\mathfrak a,\mathfrak f)=(1)} \frac{1}{\norm(\mathfrak a)} \\
&= \frac{\phi(\mathfrak f)}{\norm(\mathfrak f)}
\biggl\{(\gamma_{-1}(K)\log x+\gamma_0(K)+\gamma_{-1}(K)\eta(\mathfrak f))
\sum_{\norm(\mathfrak m)\le x}\frac{h(\mathfrak m)}{\norm(\mathfrak m)}\\
&\hspace{20mm}-\gamma_{-1}(K)
\sum_{\norm(\mathfrak m)\le x}\frac{h(\mathfrak m)\log\norm(\mathfrak m)}{\norm(\mathfrak m)}\biggr\}\\
&\quad+\bigO\!\left(\mathcal E_2(K,\mathfrak f)x^{-1/[K:\mathbb Q]}
\sum_{\norm(\mathfrak m)\le x}
\frac{|h(\mathfrak m)|}{\norm(\mathfrak m)^{1-1/[K:\mathbb Q]}}\right).
\end{aligned}
\end{equation}
Because \(\kappa<1/[K:\mathbb Q]\), the sum in the error is
\(\ll x^{1/[K:\mathbb Q]-\kappa}\).

For instance,
	\[
		\begin{aligned}
			\sum_{\norm(\mathfrak{m}) \le x} \frac{h(\mathfrak{m})}{\norm(\mathfrak{m})}
			&= \sum_{\mathfrak{m}} \frac{h(\mathfrak{m})}{\norm(\mathfrak{m})} - \sum_{\norm(\mathfrak{m}) > x} \frac{h(\mathfrak{m})}{\norm(\mathfrak{m})} \\
			&= H(1;g) - \sum_{\norm(\mathfrak{m}) > x} \frac{h(\mathfrak{m})}{\norm(\mathfrak{m})^{1-\kappa} \norm(\mathfrak{m})^\kappa} \\
			&= H(1;g) + \mathsf{O}_{\delta}(x^{-\kappa}).
		\end{aligned}
	\]
    Similarly
    \[
        \begin{aligned}
            \sum_{\norm(\mathfrak m)\le x}
\frac{h(\mathfrak m)\log\norm(\mathfrak m)}{\norm(\mathfrak m)}
&=-H'(1;g)+\bigO(x^{-\kappa}\log(2x)), \\
\sum_{\norm(\mathfrak m)\le x}
\frac{|h(\mathfrak m)|}{\norm(\mathfrak m)^{1-1/[K:\mathbb Q]}} &\ll x^{1/[K:\mathbb Q]-\kappa}.
        \end{aligned}
    \]
Substituting these estimates into
\eqref{eq:weighted reciprocal coprime mean value expansion} proves the result.
The factor involving \(\eta(\mathfrak f)\) is absorbed by
\(\eta(\mathfrak f)\ll_K\log_2 \norm(\mathfrak f) \), which follows from
the Mertens estimate \eqref{eq:uniform mertens log over norm}.
\end{proof}

\begin{lemma}\label{conrey2}
Suppose \([K:\mathbb Q]\ge2\), fix \(B>0\), and let \(g\) satisfy the
hypotheses of Lemma~\ref{lem:dirichlet-factorization}. For every fixed
integer \(j\ge0\), as \(x\to\infty\), uniformly for
\(\norm(\mathfrak f)\le x^B\),
\[
\ssum_{\norm(\mathfrak l)\le x\\(\mathfrak l,\mathfrak f)=1}
\frac{g(\mathfrak l)}{\norm(\mathfrak l)}
\left(\log\frac{x}{\norm(\mathfrak l)}\right)^j
=C_{K,\mathfrak f,\mathfrak g}\frac{\log^{j+1}x}{j+1}
+D_{K,\mathfrak f,\mathfrak g}\log^jx+\bigO(\log^jx).
\]
The constants \(C_{K,\mathfrak f,\mathfrak g}\) and
\(D_{K,\mathfrak f,\mathfrak g}\) are those of
Lemma~\ref{lem:weighted reciprocal mean value}. The implied constant
depends only on \(K,B,\delta,j\) and the uniform constant in the hypothesis
on \(g\). In particular, this applies at \(x=M=\norm(\mathfrak f)^\theta\)
for every fixed \(\theta>0\).
\end{lemma}

\begin{proof}
Write \(G(M;j)\) for the sum with cutoff \(M\).
Lemma~\ref{lem:dirichlet-factorization} gives
\(H(1;g),H'(1;g)\ll1\). Hence
\(C_{K,\mathfrak f,\mathfrak g}\ll1\).
We also have \(D_{K,\mathfrak f,\mathfrak g}\ll1+\eta(\mathfrak f)\),
which is \(\ll\log_2M\) when \(\norm(\mathfrak f)\le M^B\).
The same convolution, together with
\(\sum_{\norm(\mathfrak a)\le t}\norm(\mathfrak a)^{-1}
\ll_K\log(2t)\), proves \(|G(t;0)|\ll\log(2t)\) for all \(t\ge1\).

Lemma~\ref{lem:reciprocal coprime mean value} gives
\(\log\mathcal E_2(K,\mathfrak f)
\ll_{K,B}(\log M)^{1/[K:\mathbb Q]}/\log_2M\).
Choose any fixed \(\kappa\) allowed by
Lemma~\ref{lem:weighted reciprocal mean value}. Its error tends to zero
uniformly for \(\exp((\log M)^{1/[K:\mathbb Q]})\le t\le M\), so
\begin{equation}\label{eq:uniform asymptotic formula for G for large x}
G(t;0)=C_{K,\mathfrak f,\mathfrak g}\log t
+D_{K,\mathfrak f,\mathfrak g}+\littleo(1)
\end{equation}
in this range. Taking \(t=M\) proves the case \(j=0\).
For \(j\ge1\), partial summation gives
\begin{equation}\label{eq:partial summation for G}
G(M;j)=j\int_1^M G(t;0)
\left(\log\frac Mt\right)^{j-1}\frac{\mathrm dt}{t}.
\end{equation}
The first main term integrates to
\(C_{K,\mathfrak f,\mathfrak g}\log^{j+1}M/(j+1)\).
The second contributes \(D_{K,\mathfrak f,\mathfrak g}\log^jM\).
On the upper interval, the uniform error in
\eqref{eq:uniform asymptotic formula for G for large x} contributes
\(o(\log^jM)\). Below this cutoff, subtract both main terms and use the
crude bound for \(G(t;0)\) and the bounds for the two constants to obtain
\[
\begin{aligned}
&\int_1^{\exp((\log M)^{1/[K:\mathbb Q]})}
\left|G(t;0)-C_{K,\mathfrak f,\mathfrak g}\log t-D_{K,\mathfrak f,\mathfrak g}\right|
\left(\log\frac Mt\right)^{j-1}\frac{\mathrm dt}{t}\\
&\qquad\ll(\log M)^{j-1}
\left((\log M)^{2/[K:\mathbb Q]}
+(\log_2M)(\log M)^{1/[K:\mathbb Q]}\right)
\ll\log^jM.
\end{aligned}
\]
The last inequality uses \([K:\mathbb Q]\ge2\), proving the lemma.

\end{proof}

\subsection{Lattice bounds} \label{sec:lattice bounds}

In the off-diagonal estimates we shall interpret the fractional ideals of \( K \) as lattices, so in this section we collect some lattice counting results. The Lipschitz principle (combined with Minkowski's second formula) states that for a lattice \( \Lambda \subseteq \mathbb{R}^n \), the \( n \)-dimensional ball of  radius \( \sqrt{X} \) centered at the origin contains approximately \( C_n X^{n/2}/(\lambda_1 \lambda_2 \cdots \lambda_n) \) lattice points, where \( C_n \) is the unit \( n \)-ball volume and the denominator is the product of the successive minima of \( \Lambda \). Here we give a simple version of the Lipschitz principle for \( n = 2 \); see \cite[Theorem 2.2]{widmer2012lipschitz} for one example of a general statement.

\begin{lemma}\label{lem:lattice-disk-count}
Let \(\Lambda\) be a lattice in \(\mathbb{R}^2 \) with shortest nonzero vector of length \( \lambda_1 = \min_{0 \neq w \in \Lambda} |w| \) (the first successive minimum). There are absolute constants \( C_1, C_2 > 0 \) such that for any \(z \in \mathbb{R}^2 \) and \(X > 0\), we have
\begin{align}
\#\{w\in z+\Lambda:|w|^2\le X\}
&\le C_1 \biggl( 1+\frac{X}{\lambda_1^2} \biggr), \label{eq:lattice bound with origin} \\
\#\{0\ne w\in\Lambda:|w|^2\le X\}
&\le C_2 \frac{X}{\lambda_1^2}. \label{eq:lattice bound without origin}
\end{align}
\end{lemma}

\begin{proof}
	The (open) disks of radius \(\lambda_1/2\) centered at the points \(w \in z + \Lambda\), \(|w|^2 \le X\) are pairwise disjoint and are contained in a disk of radius \( \sqrt{X} + \lambda_1/2\). Comparing volumes and applying the elementary inequality \( (x+y)^2 \le 2(x^2 + y^2) \) gives
\[
	\# \{w \in z + \Lambda : |w|^2 \le X \} \le \frac{\pi(\sqrt{X} + \lambda_1/2)^2}{\pi (\lambda_1/2)^2} \le 2 + \frac{8X}{\lambda_1^2}.
\]
For \eqref{eq:lattice bound without origin}, take \( z=0 \) in \eqref{eq:lattice bound with origin} and note that the set in \eqref{eq:lattice bound without origin} is nonempty exactly when \( X \ge \lambda_1^2 \), so the \( 1 \) in \eqref{eq:lattice bound with origin} can be absorbed in this case.
\end{proof}

\begin{numberedcorollary}\label{lem:ideal-coset-lattice-count}
Assume that \(K\) is imaginary quadratic. Let \(\mathfrak a\) be a nonzero
fractional ideal. Uniformly for \(z\in K \) and \(X>0\),
\[
\begin{aligned}
\#\{\alpha\in z+ \mathfrak{a} : \norm(\alpha) \le X\}
&\ll 1+\frac{X}{\norm(\mathfrak a)},\\
\#\{0\ne\alpha\in\mathfrak a:\norm(\alpha)\le X\}
&\ll\frac{X}{\norm(\mathfrak a)}.
\end{aligned}
\]
Here \( z + \mathfrak{a} = \{z + \alpha : \alpha \in \mathfrak{a}\} \) is the translate of \( \mathfrak{a} \).
\end{numberedcorollary}

\begin{proof} Fix an embedding \(\sigma:K\hookrightarrow\mathbb C\). The image \(\sigma(\mathfrak a)\) is a lattice in \(\mathbb C = \mathbb{R}^2 \). Since \(K\) is imaginary quadratic, the norm is given by the Euclidean distance, \(|\norm(\alpha)| = |\sigma(\alpha)|^2\). So the shortest nonzero vector in \( \sigma(\mathfrak{a}) \) has length at least \( \norm(\mathfrak{a})^{1/2} \), and then both assertions follow from Lemma~\ref{lem:lattice-disk-count}.
\end{proof}

Here we rely on the field norm being Euclidean; the authors expect that Lemma~\ref{lem:ideal-coset-lattice-count} holds for arbitrary number fields when counting principal ideals instead of elements, by the stronger version of Lipschitz's principle aforementioned. Similarly they expect analogues of the following off-diagonal estimates to hold for other number fields.

\begin{lemma}\label{lem:ray-principal-ideal-count}
Let \(K\) be imaginary quadratic and let \(\mathfrak m,\mathfrak f,\mathfrak t\)
be nonzero integral ideals satisfying \((\mathfrak m,\mathfrak f)=1\) and
\(\mathfrak t\mid\mathfrak f\). For \(X>0\),
\begin{equation}\label{eq:ray-principal-ideal-count}
\#\{0\ne\mathfrak J\subseteq\mathfrak m:
\mathfrak J\in P_K(\mathfrak t),\ \norm(\mathfrak J)\le X\}
\ll_K1+\frac{X}{\norm(\mathfrak m)\norm(\mathfrak t)}.
\end{equation}
If \(V : [0,\infty) \to \mathbb{C}\) satisfies \eqref{eq:approximation to indicator}, then
\begin{equation}\label{eq:ray-principal-ideal-weighted}
\sum_{\mathfrak t\mid\mathfrak f}\phi(\mathfrak t)
\sum_{\substack{0\ne\mathfrak J\subseteq\mathfrak m\\
\mathcal O_K\ne\mathfrak J\in P_K(\mathfrak t)}}
\frac1{\norm(\mathfrak J)^{1/2}}
\left|V\left(\frac{\norm(\mathfrak J)}X\right)\right|
\ll_{K,V}d(\mathfrak f)X^{1/2}.
\end{equation}
\end{lemma}

\begin{proof}
For each \( \mathfrak{J} \) counted in \eqref{eq:ray-principal-ideal-count}, choose a generator \(\alpha\equiv1\pmod{\mathfrak t}\). Since \( \mathfrak{m} \), \( \mathfrak{t} \) are coprime, the solutions to \( \alpha \equiv 0 \) (mod \( \mathfrak{m} \)), \( \alpha \equiv 1 \) (mod \( \mathfrak{t} \)) form a translate of \(\mathfrak{mt}\). So \eqref{eq:ray-principal-ideal-count} follows from Corollary~\ref{lem:ideal-coset-lattice-count}.

For \eqref{eq:ray-principal-ideal-weighted}, every ideal \(\mathfrak J\) in the inner sum satisfies \(\norm(\mathfrak t)\le4\norm(\mathfrak J)\). Indeed, the off-diagonal condition
\(\mathfrak J\ne\mathcal O_K\) implies \(\alpha-1\ne0\).
Since \(\alpha-1\in\mathfrak t\) and \(\norm(\mathfrak J)\ge1\),
\[
\norm(\mathfrak t)\le\norm(\alpha-1)
\le(\sqrt{\norm(\mathfrak J)}+1)^2
\le4\norm(\mathfrak J).
\]
Now partial the inner sum into dyadic intervals
\(2^j\norm(\mathfrak m)\le\norm(\mathfrak J)
<2^{j+1}\norm(\mathfrak m)\), for \(j\ge0\). For each fixed \(\mathfrak t\), the contribution of a block is at most
\[
\ll_{K,V,A}
\left(1+\frac{2^j}{\norm(\mathfrak t)}\right)
\frac1{2^{j/2}\norm(\mathfrak m)^{1/2}}
\min\left(1,\left(\frac{2^j\norm(\mathfrak m)}X\right)^{-A}\right)
\]
by \eqref{eq:ray-principal-ideal-count} and the decay of \(V\). In this block we have \(\norm(\mathfrak{t}) < 2^{j+3} \norm(\mathfrak{m}) \), and
\[
\sum_{\substack{\mathfrak t\mid\mathfrak f\\
\norm(\mathfrak t)<2^{j+3}\norm(\mathfrak m)}}
\phi(\mathfrak t)\left(1+\frac{2^j}{\norm(\mathfrak t)}\right)
\ll2^jd(\mathfrak f)\norm(\mathfrak m).
\]
Thus the total contribution of the block is
\(\ll_{K,V,A}d(\mathfrak f)(2^j\norm(\mathfrak m))^{1/2}
\min(1,(2^j\norm(\mathfrak m)/X)^{-A})\).
The sum over all \( j \ge 0 \) is readily computed by splitting the sum at \( 2^j \norm(\mathfrak{m}) = X \) and choosing \(A>1/2\).
\end{proof}

\begin{lemma}\label{lem:ray-ideal-pair-lattice}
Let \(K\) be imaginary quadratic and let \(\mathfrak t\) be a nonzero
integral ideal. For \(X,Y\ge1\), we have the following count of pairs of integral ideals.
\begin{equation}\label{eq:ideal-pair-lattice-count}
\#\left\{(\mathfrak J_1,\mathfrak J_2):
\begin{array}{c}
0\ne\mathfrak J_1,\mathfrak J_2\subseteq\mathcal O_K,
\quad\mathcal O_K\ne\mathfrak J_1\mathfrak J_2^{-1}\in P_K(\mathfrak t)\\
\norm(\mathfrak J_1)\le X,\quad\norm(\mathfrak J_2)\le Y
\end{array}\right\}
\ll_K\frac{XY}{\norm(\mathfrak t)}.
\end{equation}
Suppose that \(W : [0,\infty) \to \mathbb{C} \) satisfies \eqref{eq:approximation to indicator}.
For \(X\ge1\) and \(0\le\eta<1/2\),
\begin{equation}\label{eq:ideal-pair-lattice-weighted}
\sum_{\mathcal O_K\ne\mathfrak J_1\mathfrak J_2^{-1}\in P_K(\mathfrak t)}
\norm(\mathfrak J_1\mathfrak J_2)^{-1/2+\eta}
\left|W\left(\frac{\norm(\mathfrak J_1\mathfrak J_2)}{X^2}\right)\right|
\ll_{K,W,\eta}\frac{X^{1+2\eta}\log(2X)}{\norm(\mathfrak t)}.
\end{equation}
(In fact, it is enough to satisfy those estimates for just one \(A > 1/2 + \eta\).)
\end{lemma}

\begin{proof}
Exchanging \(\mathfrak J_1\) and \(\mathfrak J_2\) preserves the ray
condition, so we may assume \(X\ge Y\).
For each pair counted, choose a generator
\(\mathfrak J_1\mathfrak J_2^{-1}=(\alpha)\) with
\(\alpha\equiv1\pmod{\mathfrak t}\).
The element \(\alpha\) may be fractional, and
\(\alpha-1\ne0\) because \(\mathfrak J_1\mathfrak J_2^{-1}\ne\mathcal O_K\).
Moreover, \((\alpha-1)\mathfrak J_2\subseteq
\mathfrak J_1+\mathfrak J_2\subseteq\mathcal O_K\).
At every prime dividing \(\mathfrak t\), the congruence for \(\alpha\)
and the integrality of \(\mathfrak J_2\) give the additional divisibility
by \(\mathfrak t\). Thus
\(\alpha-1\in\mathfrak t\mathfrak J_2^{-1}\).
Since \(\norm(\mathfrak J_2)\le Y\le X\),
\[
\norm(\alpha-1)
\le\left(\sqrt{\frac{\norm(\mathfrak J_1)}{\norm(\mathfrak J_2)}}+1\right)^2
\le\frac{4X}{\norm(\mathfrak J_2)}.
\]
For each fixed \(\mathfrak J_2\), distinct \(\mathfrak J_1\) give distinct
chosen elements \(\alpha-1\). The nonzero lattice-point estimate in
Corollary~\ref{lem:ideal-coset-lattice-count} therefore bounds their number by
\[
\#\left\{0\ne\alpha-1\in\mathfrak t\mathfrak J_2^{-1}:
\norm(\alpha-1)\le\frac{4X}{\norm(\mathfrak J_2)}\right\}
\ll_K\frac{X/\norm(\mathfrak J_2)}{\norm(\mathfrak t\mathfrak J_2^{-1})}
=\frac{X}{\norm(\mathfrak t)}.
\]
There are \(\ll_KY\) choices of \(\mathfrak J_2\) by
Lemma~\ref{lem:ideal-density-mertens}, proving
\eqref{eq:ideal-pair-lattice-count}.

For \eqref{eq:ideal-pair-lattice-weighted}, divide both ideal norms into
dyadic intervals starting at \(1\).
The block \(2^i\le\norm(\mathfrak J_1)<2^{i+1}\),
\(2^j\le\norm(\mathfrak J_2)<2^{j+1}\), contributes at most
\[
\ll_{W,A,\eta}\frac{2^{(1/2+\eta)(i+j)}}{\norm(\mathfrak t)}
\min\left(1,\left(\frac{2^{i+j}}{X^2}\right)^{-A}\right)
\]
by \eqref{eq:ideal-pair-lattice-count} and the decay of \(W\).
For each \(k=i+j\) there are \(k+1\) blocks, so their total is bounded by
\[
\frac1{\norm(\mathfrak t)}
\sum_{k\ge0}(k+1)2^{(1/2+\eta)k}
\min\left(1,\left(\frac{2^k}{X^2}\right)^{-A}\right).
\]
We obtain the desired bound by by splitting at \(2^k=X^2\) and choosing \(A>1/2+\eta\).
\end{proof}

\section{Twisted first moment}\label{sec:first-moment}
Henceforth we assume that \(K\) is an imaginary quadratic field and \(\norm(\mathfrak f)>1\). Thus every primitive character modulo \(\mathfrak f\) is nontrivial.

\subsection{Some character sums} \label{sec:character sums}

\begin{lemma}
\label{lem:gauss-orthogonality}
Let \(\tau(\chi)\) be the Gauss sum occurring in the functional
equation of Lemma~\ref{thm:functional-equation}.
Then, for every nonzero integral ideal \(\mathfrak m\) satisfying
\((\mathfrak m,\mathfrak f)=(1)\),
\[
\begin{aligned}
\sum_{\chi\bmod\mathfrak f}^{*}
\chi(\mathfrak m)\tau(\chi)
&=
\sum_{\mathfrak s\mathfrak t=\mathfrak f}
\mu(\mathfrak s)h_K(\mathfrak t)
\sum_{\substack{\alpha\bmod\mathfrak f\\
(\alpha,\mathfrak f)=(1)\\
\mathfrak m\mathfrak e(\alpha)\in P_K(\mathfrak t)}}
\exp\!\left(
2\pi i\Tr(\alpha\beta)
\right).
\end{aligned}
\]
In particular, if \(\mathfrak m\mathfrak e\) is not principal, then
the right-hand side vanishes.
\end{lemma}

\begin{proof}
Expanding \(\tau(\chi)\) and interchanging the finite sums, the left-hand
side becomes
\[
\sum_{\substack{\alpha\bmod\mathfrak f\\
(\alpha,\mathfrak f)=(1)}}
e^{2\pi i\Tr(\alpha\beta)}
\sum_{\chi\bmod\mathfrak f}^{*}
\chi\!\left(\mathfrak m\mathfrak e(\alpha)\right).
\]

Fix \(\alpha\).  The fractional ideal
\(\mathfrak m\mathfrak e(\alpha)\) is coprime to \(\mathfrak f\).
By clearing its denominator away from \(\mathfrak f\) and applying the
Chinese remainder theorem, we may choose
\(\gamma\in\mathcal O_K\) such that
\(\gamma\equiv1\pmod{\mathfrak f}\) and
\((\gamma)\mathfrak m\mathfrak e(\alpha)\) is an integral ideal coprime
to \(\mathfrak f\).

Since
\((\gamma)\in P_K(\mathfrak f)\subseteq P_K(\mathfrak t)\) for every
\(\mathfrak t\mid\mathfrak f\), multiplication by \((\gamma)\) changes
neither the value of a character modulo \(\mathfrak f\) nor membership
in \(P_K(\mathfrak t)\).  Hence
Lemma~\ref{lem:primitive-ray-orthogonality}, applied to
\((\gamma)\mathfrak m\mathfrak e(\alpha)\), gives
\[
\sum_{\chi\bmod\mathfrak f}^{*}
\chi\!\left(\mathfrak m\mathfrak e(\alpha)\right)
=
\sum_{\substack{\mathfrak s\mathfrak t=\mathfrak f\\
\mathfrak m\mathfrak e(\alpha)\in P_K(\mathfrak t)}}
\mu(\mathfrak s)h_K(\mathfrak t).
\]
Substitution proves the asserted identity.

Finally, if
\(\mathfrak m\mathfrak e(\alpha)\in P_K(\mathfrak t)\), then
\(\mathfrak m\mathfrak e(\alpha)\) is principal, and therefore so is
\(\mathfrak m\mathfrak e\), since \((\alpha)\) is principal.  Thus, if
\(\mathfrak m\mathfrak e\) is not principal, no such \(\alpha\) exists
and the right-hand side vanishes.
\end{proof}

\begin{lemma}
\label{lem:gauss-sum-bound}
Under the assumptions of Lemma~\ref{lem:gauss-orthogonality}, for every
nonzero integral ideal \(\mathfrak m\) satisfying
\((\mathfrak m,\mathfrak f)=(1)\),
\[
\left|
\sum_{\chi\bmod\mathfrak f}^{*}
\chi(\mathfrak m)\tau(\chi)
\right|
\le
|\mathcal O_K^\times|
\sum_{\mathfrak t\mid\mathfrak f}
h_K(\mathfrak t)
\frac{\norm(\mathfrak f)}{\norm(\mathfrak t)}
\ll_K
d(\mathfrak f)\norm(\mathfrak f),
\]
where
 \(d(\mathfrak f)\) denotes the number of integral ideal divisors
of \(\mathfrak f\).
\end{lemma}

\begin{proof}
By Lemma~\ref{lem:gauss-orthogonality}, the triangle inequality, and
\(|\mu(\mathfrak s)|\le1\), the left-hand side is bounded by
\[
\sum_{\mathfrak t\mid\mathfrak f}
h_K(\mathfrak t)
\#
\left\{
\alpha\bmod\mathfrak f:
(\alpha,\mathfrak f)=(1),\
\mathfrak m\mathfrak e(\alpha)\in P_K(\mathfrak t)
\right\}.
\]

If \(\mathfrak m\mathfrak e\) is not principal, all the sets in
the preceding display are empty by
Lemma~\ref{lem:gauss-orthogonality}.  We may therefore assume that
\(\mathfrak m\mathfrak e=(\xi)\) for some \(\xi\in K^\times\).

Since
\(
(\mathfrak m\mathfrak e,\mathfrak f)=(1),
\)
we have
\(
v_{\mathfrak p}(\xi)=0,
(\mathfrak p\mid\mathfrak f).
\)
Thus \(\xi\) is invertible modulo every
\(\mathfrak t\mid\mathfrak f\). The condition
\(
\mathfrak m\mathfrak e(\alpha)
=
(\xi\alpha)
\in P_K(\mathfrak t)
\)
is equivalent to the existence of a unit
\(u\in\mathcal O_K^\times\) such that
\(\xi\alpha\equiv u\pmod{\mathfrak t}\).
Indeed, if \((\xi\alpha)\in P_K(\mathfrak t)\), then there exists
\(\lambda\equiv1\pmod{\mathfrak t}\) satisfying
\(
(\lambda)=(\xi\alpha).
\)
Hence
\(
\lambda=u\xi\alpha
\)
for some \(u\in\mathcal O_K^\times\), and therefore
\(
\xi\alpha\equiv u^{-1}\pmod{\mathfrak t}.
\)
Since inverses of units are again units, this is equivalent to
the preceding congruence.  The converse follows in the same way.

For each fixed unit \(u\), since \(\xi\) is invertible modulo
\(\mathfrak t\), this congruence determines the single residue class
\(\alpha\equiv\xi^{-1}u\pmod{\mathfrak t}\).
The natural reduction map
\(\mathcal O_K/\mathfrak f\longrightarrow\mathcal O_K/\mathfrak t\)
has fibers of cardinality
\(
\frac{\norm(\mathfrak f)}{\norm(\mathfrak t)}.
\)
Restricting to residue classes coprime to \(\mathfrak f\) can only
decrease the number of lifts.  Hence
\[
\#
\left\{
\alpha\bmod\mathfrak f:
(\alpha,\mathfrak f)=(1),\
\mathfrak m\mathfrak e(\alpha)\in P_K(\mathfrak t)
\right\}
\le
|\mathcal{O}_{K}^{\times}|
\frac{\norm(\mathfrak f)}{\norm(\mathfrak t)}.
\]
Substituting this into the earlier absolute-value bound gives
\[
\left|
\sum_{\chi\bmod\mathfrak f}^{*}
\chi(\mathfrak m)\tau(\chi)
\right|
\le
|\mathcal O_K^\times|
\sum_{\mathfrak t\mid\mathfrak f}
h_K(\mathfrak t)
\frac{\norm(\mathfrak f)}{\norm(\mathfrak t)}.
\]

Finally, the cardinality formula in
Lemma~\ref{lem:ray-class-exact-sequence} gives
\[
h_K(\mathfrak t)
=
h_K
\frac{\phi(\mathfrak t)}
{[\mathcal O_K^\times:
  \mathcal O_K^\times\cap K^{\mathfrak t,1}]}
\le
h_K\phi(\mathfrak t)
\le
h_K\norm(\mathfrak t).
\]
Therefore
\[
\begin{aligned}
|\mathcal O_K^\times|
\sum_{\mathfrak t\mid\mathfrak f}
h_K(\mathfrak t)
\frac{\norm(\mathfrak f)}{\norm(\mathfrak t)}
&\le
|\mathcal O_K^\times|h_K\norm(\mathfrak f)
\sum_{\mathfrak t\mid\mathfrak f}1 \\
&=
|\mathcal O_K^\times|h_Kd(\mathfrak f)\norm(\mathfrak f) \\
&\ll_K
d(\mathfrak f)\norm(\mathfrak f).
\end{aligned}
\]
This proves the lemma.
\end{proof}

\begin{numberedcorollary}\label{cor:gauss-sum-bound-fractional}
For every fractional ideal
\(\mathfrak b\in I_K(\mathfrak f)\),
\[
\left|
\sum_{\chi\bmod\mathfrak f}^{*}
\chi(\mathfrak b)\tau(\chi)
\right|
\le
|\mathcal O_K^\times|
\sum_{\mathfrak t\mid\mathfrak f}
h_K(\mathfrak t)
\frac{\norm(\mathfrak f)}{\norm(\mathfrak t)}
\ll_K
d(\mathfrak f)\norm(\mathfrak f).
\]
\end{numberedcorollary}

\begin{proof}
Choose an integral ideal \(\mathfrak d\), coprime to
\(\mathfrak f\), such that \(\mathfrak d\mathfrak b\) is integral.
By the Chinese remainder theorem, choose
\(\gamma\in\mathfrak d\) with
\(\gamma\equiv1\pmod{\mathfrak f}\). Then
\(\mathfrak n:=(\gamma)\mathfrak b\) is an integral ideal coprime to
\(\mathfrak f\). Since \((\gamma)\in P_K(\mathfrak f)\), one has
\(\chi(\mathfrak n)=\chi(\mathfrak b)\)
for every character modulo \(\mathfrak f\). The result now follows by
applying Lemma~\ref{lem:gauss-sum-bound} to \(\mathfrak n\).
\end{proof}

\subsection{Twisted first moment}
For a nonzero integral ideal \(\mathfrak m\) coprime to \(\mathfrak f\), the twisted first moment is
\[
	\mathcal{A}(\m) \coloneqq \sideset{}{^*}\ssum_{\chi \bmod \f} \chi(\m)\, L\bigl(\tfrac{1}{2},\chi\bigr).
\]
The approximate functional equation, Lemma~\ref{approximate functional equation}, gives
\[
	\mathcal{A}(\m) = \sideset{}{^*}\ssum_{\chi \text{ mod } \mathfrak{f}} \chi(\m)\!\biggl[ \ssum_{(\a,\f) = 1} \frac{\chi(\a)}{\norm(\a)^{1/2}}\, V\!\left(\frac{\norm(\a)}{\norm(\f)^{1/2}}\right) + \frac{\tau(\chi)}{\norm(\f)^{1/2}} \ssum_{(\a,\f) = 1} \frac{\overline{\chi(\a)}}{\norm(\a)^{1/2}}\, V\!\left(\frac{\norm(\a)}{\norm(\f)^{1/2}}\right) \biggr].
\]
Interchanging sums and summing over characters with
Lemma~\ref{lem:primitive-ray-orthogonality} yields
\(\mathcal{A}(\m)=\mathcal{A}_1(\m)+\mathcal{A}_2(\m)\) where
\begin{align*}
	\mathcal{A}_1(\m) &\coloneqq \ssum_{(\a,\f) = 1} \frac{1}{\norm(\a)^{1/2}}\, V\!\left(\frac{\norm(\a)}{\norm(\f)^{1/2}}\right) \ssum_{\substack{\a\m \in P_K(\t) \\ \s\t = \f}} \mu(\s)\, h_K(\t), \\
	\mathcal{A}_2(\m) &\coloneqq \frac{1}{\norm(\f)^{1/2}} \ssum_{(\a,\f) = 1} \frac{1}{\norm(\a)^{1/2}}\, V\!\left(\frac{\norm(\a)}{\norm(\f)^{1/2}}\right) \sideset{}{^*}\ssum_{\chi \text{ mod } \mathfrak{f}} \tau(\chi)\, \chi(\a^{-1}\m).
\end{align*}

\subsection{\texorpdfstring{Main term of \(\mathcal A_1\)}{Main term of A1}}
\label{sec:first-moment-main-term}

The diagonal contribution to \(\mathcal A_1(\mathfrak m)\) comes from \(\mathfrak a\mathfrak m=\mathcal O_K\), which for integral \(\mathfrak a,\mathfrak m\) means \(\mathfrak a=\mathfrak m=\mathcal O_K\). By Lemma~\ref{lem:primitive-ray-orthogonality} and \eqref{eq:approximation to indicator}, this term is
\[
	\mathcal{A}_{1,\mathrm{main}}(\mathfrak m) = \delta_{\mathfrak m}h_K^*(\mathfrak f) V \biggl(\frac{1}{\norm(\mathfrak{f})^{1/2}}\biggr) = \delta_{\mathfrak m} h_K^*(\mathfrak{f}) (1 + \mathsf{O}_A(\norm(\mathfrak{f})^{-A/2}))
\]
where \( \delta_\mathfrak{m} = 1 \) if \( \mathfrak{m} = \mathcal{O}_K \) and \( \delta_\mathfrak{m} = 0 \) otherwise.

\subsection{\texorpdfstring{Error term of \(\mathcal A_1\)}{Error term of A1}}

The remaining terms contribute
\[
	\mathcal{A}_1(\mathfrak{m}) - \mathcal{A}_{1,\text{main}}(\mathfrak{m}) = \sum_{(\a,\f) = 1} \frac{1}{\norm(\a)^{1/2}}\, V\!\left(\frac{\norm(\a)}{\norm(\f)^{1/2}}\right) \ssum_{\substack{\mathcal{O}_K \neq \a\m \in P_K(\t) \\ \s\t = \f}} \mu(\s)\, h_K(\t).
\]
The inner sum is supported only on ideals \( \mathfrak{a} \) that make \( \mathfrak{am} \) principal. We re-index by replacing \( \mathfrak{a} \) with \( \mathfrak{J} = \mathfrak{am} \in P_K(\mathfrak{t}) \). Bounding \( \mu(\mathfrak{s}) \) trivially and using \( h_K(\mathfrak{t}) \le h_K \phi(\mathfrak{t}) \) (Lemma~\ref{lem:ray-class-exact-sequence}) gives
\[
	\mathcal{A}_1(\mathfrak{m}) - \mathcal{A}_{1,\text{main}}(\mathfrak{m}) \ll_K \norm(\mathfrak{m})^{1/2} \sum_{\mathfrak{t} \mid \mathfrak{f}} \phi(\mathfrak{t}) \ssum_{\mathcal{O}_K \neq \mathfrak{J} \in P_K(\mathfrak{t}) \\ \mathfrak{J} \subseteq \mathfrak{m}} \frac{1}{\norm(\mathfrak{J})^{1/2}} \biggl| V \biggl(\frac{\norm(\mathfrak{J})}{\norm(\mathfrak{m}) \norm(\mathfrak{f})^{1/2}}\biggr)\biggr|,
\]
so by Lemma~\ref{lem:ray-principal-ideal-count} with \( X = \norm(\mathfrak{m}) \norm(\mathfrak{f})^{1/2} \),
\begin{equation} \label{A1 error}
	\mathcal{A}_1(\mathfrak{m}) - \mathcal{A}_{1,\text{main}}(\mathfrak{m}) \ll \norm(\mathfrak{m}) \norm(\mathfrak{f})^{1/4+\varepsilon}.
\end{equation}

\subsection{\texorpdfstring{Estimation of \(\mathcal A_2\)}{Estimation of A2}}

We apply Corollary~\ref{cor:gauss-sum-bound-fractional} to the fractional ideal \(\mathfrak a^{-1}\mathfrak m\), which is coprime to \(\mathfrak f\), to get
\[
	|\mathcal{A}_2(\m)| \ll_K d(\f)\, \norm(\f)^{1/2} \, \ssum_{(\a,\f) = (1)} \frac{1}{\sqrt{\norm(\a)}}\, \left|V\!\left(\frac{\norm(\a)}{\sqrt{\norm(\f)}}\right)\right|.
\]
Split the sum over $\a$ at $\norm(\a) = \norm(\mathfrak{f})^{1/2}$. Write \(I_1\) for the sum over \( \norm{\mathfrak{a}} \le \norm(\mathfrak{f})^{1/2}\), and \(I_2\) for \( \norm{\mathfrak{a}} > \norm(\mathfrak{f})^{1/2}\). For \( I_1 \), we bound \( V \) trivially. Then partial summation and the ideal counting estimate in Lemma~\ref{lem:ideal-density-mertens} yield
\[
	I_1 \ll \norm(\mathfrak{f})^{1/2} \times \norm(\f)^{-1/4} + \int_1^{\norm(\mathfrak{f})^{1/2}} t \times t^{-3/2}\, dt
	\ll_K \norm(\f)^{1/4}.
\]
For the tail, choose \(A=1\) in \eqref{eq:approximation to indicator} and apply partial summation again, giving
\[
I_2
\ll_K
\norm(\mathfrak f)^{1/2}
\sum_{\norm(\mathfrak a)>\sqrt{\norm(\mathfrak f)}}
\norm(\mathfrak a)^{-3/2}
\ll_K \norm(\mathfrak f)^{1/4}.
\]
We conclude that
\begin{equation} \label{A2 bound}
	|\mathcal{A}_2(\m)| \ll_K d(\f)\, \norm(\f)^{3/4}
	\ll \norm(\f)^{3/4+\varepsilon}.
\end{equation}

\section{Mollified first moment}\label{sec:mollified-first-moment}
In this section we prove Proposition~\ref{prop:first mollified moment}. The mollified first moment expands as
\[
\mathcal C(M)
\coloneqq
\sideset{}{^*}\sum_{\chi\bmod\mathfrak f}
\mathcal M(\chi)L\!\left(\frac12,\chi\right)
=
\sum_{\substack{\norm(\mathfrak m)\le M\\
(\mathfrak m,\mathfrak f)=(1)}}
\frac{\widetilde\mu(\mathfrak m)}{\norm(\mathfrak{m})^{1/2}}
\mathcal A(\mathfrak m),
\]
where \(\mathcal A(\mathfrak m)=\mathcal A_1(\mathfrak m)+\mathcal A_2(\mathfrak m)\) is the twisted first moment computed in \Cref{sec:first-moment}. Taking \(\mathfrak{m} = \mathcal{O}_K\), we see that the main term of \( \mathcal{C}(M) \) is \( \mathcal C_{\mathrm{main}}(M)
=
h_K^*(\mathfrak f)
\left(1+\littleo(1)\right) \).
The remaining contribution is
\(\mathcal C(M)-\mathcal C_{\mathrm{main}}(M)
=\mathcal E_1(M)+\mathcal E_2(M)\) where
\[
\begin{aligned}
\mathcal E_1(M)
&:=
\sum_{\substack{\norm(\mathfrak m)\le M\\
(\mathfrak m,\mathfrak f)=(1)}}
\widetilde\mu(\mathfrak m)\norm(\mathfrak m)^{-1/2}
\bigl(
\mathcal A_1(\mathfrak m)
-
\mathcal A_{1,\mathrm{main}}(\mathfrak m)
\bigr),\\
\mathcal E_2(M)
&:=
\sum_{\substack{\norm(\mathfrak m)\le M\\
(\mathfrak m,\mathfrak f)=(1)}}
\widetilde\mu(\mathfrak m)\norm(\mathfrak m)^{-1/2}
\mathcal A_2(\mathfrak m).
\end{aligned}
\]
With \eqref{A1 error} and \eqref{A2 bound} we obtain
\[
|\mathcal E_1(M)|
\ll_{K,\varepsilon}
\norm(\mathfrak f)^{1/4+\varepsilon}
\sum_{\norm(\mathfrak m)\le M}
\norm(\mathfrak m)^{1/2}
\ll
M^{3/2}\norm(\mathfrak f)^{1/4+\varepsilon}
\]
and
\[
|\mathcal E_2(M)|
\ll_{K,\varepsilon}
\norm(\mathfrak f)^{3/4+\varepsilon}
\sum_{\norm(\mathfrak m)\le M}
\norm(\mathfrak m)^{-1/2}
\ll
M^{1/2}\norm(\mathfrak f)^{3/4+\varepsilon}.
\]
Hence we have
\begin{equation}
\label{C asymptotic first moment}
\mathcal C(M)
=
h_K^*(\mathfrak f)(1+\mathsf{O}_A(\norm(\mathfrak{f})^{-A/2}))
+
\bigO \!\left(
M^{3/2}\norm(\mathfrak f)^{1/4+\varepsilon}
+
M^{1/2}\norm(\mathfrak f)^{3/4+\varepsilon}
\right),
\end{equation}
which is Proposition~\ref{prop:first mollified moment}.

\section{Twisted second moment}\label{sec:second-moment}
For nonzero integral ideals \(\mathfrak m_1,\mathfrak m_2\) coprime to \(\mathfrak f\), define the twisted second moment by
\[
\mathcal{B}(\m_1,\m_2)=\sum^{*}_{\chi \mod \f}\chi(\m_1)\overline{\chi}(\m_2)|L(\tfrac{1}{2},\chi)|^2.
\]
By Lemmas~\ref{approximate functional equation II} and
\ref{lem:primitive-ray-orthogonality}, we get
\begin{equation}\label{second moments}
\begin{aligned}
\mathcal{B}(\m_1,\m_2)
={}&2\sum_{n}W\left(\frac{4n\pi^2}{|d_{K}|\norm(\f)}\right)n^{-1/2}
\sum_{\substack{\norm(\a_1\a_2)=n\\
(\a_1\a_2\m_1\m_2,\f)=1}}
\sum_{\substack{\mathfrak{a}_1 \mathfrak{m}_1
\mathfrak{a}_2^{-1} \mathfrak{m}_2^{-1} \in P_{K}(\t)\\
\s\t=\f}}
\mu(\s)h_{K}(\t).
\end{aligned}
\end{equation}

\subsection{Diagonal contribution}\label{sec:twisted-diagonal}

This calculation follows the proof of \cite[Lemma~4.1]{Iwaniec-Sarnak}.
The diagonal terms in \eqref{second moments} are those satisfying \(\mathfrak a_1\mathfrak m_1=\mathfrak a_2\mathfrak m_2\). Write \(\mathcal B_{\mathrm{main}}(\mathfrak m_1,\mathfrak m_2)\) for their contribution. For every integral ideal \(\mathfrak l\) coprime to \(\mathfrak f\), the identity \(\chi(\mathfrak l)\overline{\chi(\mathfrak l)}=1\) gives \(\mathcal B(\mathfrak l\mathfrak n_1,\mathfrak l\mathfrak n_2)=\mathcal B(\mathfrak n_1,\mathfrak n_2)\), and the same identity holds for the diagonal contribution. Taking \(\mathfrak l=(\mathfrak m_1,\mathfrak m_2)\), it therefore suffices to consider coprime \(\mathfrak n_1,\mathfrak n_2\). The diagonal condition then gives \(\mathfrak a_1=\mathfrak b\mathfrak n_2\) and \(\mathfrak a_2=\mathfrak b\mathfrak n_1\) for a unique integral ideal \(\mathfrak b\), so
\begin{equation} \label{eq:S formula for B main}
\begin{aligned}
\mathcal B_{\mathrm{main}}(\mathfrak n_1,\mathfrak n_2)
&=2h_K^*(\mathfrak f)
\sum_{n=1}^{\infty}
W\!\left(\frac{4\pi^2n}{|d_K|\norm(\mathfrak f)}\right)n^{-1/2}
\ssum_{\substack{\mathfrak b:\,\norm(\mathfrak b^2\mathfrak n_1 \mathfrak n_2 )=n\\
(\mathfrak b,\mathfrak f)=1}}1\\
&=\frac{2h_K^*(\mathfrak f)}
{\norm(\mathfrak n_1\mathfrak n_2)^{1/2}}
S\!\left(\frac{|d_K|\norm(\mathfrak f)}{
4\pi^2\norm(\mathfrak n_1 \mathfrak n_2 )}\right),
\end{aligned}
\end{equation}
where
\[
S(X)=\sum_{n=1}^{\infty}
\frac{c_\mathfrak f(n)}{n}W\!\left(\frac{n^2}{X}\right), \qquad c_\mathfrak{f}(n) = \# \{\mathfrak{b} \subseteq \mathcal{O}_K : (\mathfrak{b}, \mathfrak{f}) = 1, \norm(\mathfrak{b}) = n\}.
\]
Note that the Dirichlet series associated to \(c_\mathfrak{f}(n)\) is simply \(
	\sum_{n=1}^\infty \frac{c_\mathfrak{f}(n)}{n^s} = \zeta_{K,\mathfrak{f}}(s)
\).
By Mellin inversion,
\[
	S(X) = \frac{1}{2 \pi \mathsf{i}} \int_{(3)} \Gamma^2(s + \tfrac{1}{2}) G^2(s) \zeta_{K,\mathfrak{f}}(1+2s) X^s \frac{\mathrm{d}s}{s}.
\]
    For the following computations we choose the weight function \( G \) defining \(W\) in Lemma~\ref{approximate functional equation II} to be normalized by
\(
G(0)\Gamma\!\left(\frac12\right)=1
\) and  \(G(-\frac{1}{2}) = G'(0) = 0 \). Then the integrand is meromorphic in
	\(\Re s>-1\) with a double pole at \(s=0\) and no other singularities. Shifting the contour to
	\(\Re s=-\frac12\) gives
\begin{equation} \label{eq:S shifted formula}
    S(X) = \operatorname{res}(s=0) + \bigO_K\!\left(d(\mathfrak{f}) X^{-1/2}\right)
\end{equation}
	since
	\(
	\left|
	\prod_{\mathfrak p\mid\mathfrak f}
	\left(1-\norm(\mathfrak p)^{-(1+2s)}\right)
	\right|
	\le
	\prod_{\mathfrak p\mid\mathfrak f}2
	\le d(\mathfrak f)
	\)
    on the new contour. The Laurent expansions at \(s=0\) are given by
\[
    \begin{aligned}
	s^{-1} \Gamma^2(s + \tfrac{1}{2}) G^2(s) X^s
        &= s^{-1} + \log X + 2 \frac{\Gamma'}{\Gamma}\biggl(\frac{1}{2}\biggr) + \mathsf{O}(s), \\
    \zeta_{K,\mathfrak{f}}(1+2s)
        &= \frac{\phi(\mathfrak{f})}{\norm(\mathfrak{f})} \biggl(\frac{\gamma_{-1}(K)}{2s} + \gamma_0(K) + \gamma_{-1}(K) \eta(\mathfrak{f}) + \mathsf{O}(s)\biggr).
    \end{aligned}
\]
Thus the main term of \(S(X)\) is
\[
	\operatorname{res}(s=0) = \frac{\phi(\mathfrak{f})}{\norm(\mathfrak{f})}  \biggl( \frac{\gamma_{-1}(K)}{2} \log X + \gamma_{-1}(K) \frac{\Gamma'}{\Gamma}\biggl(\frac{1}{2}\biggr) + \gamma_0(K) + \gamma_{-1}(K)\eta(\mathfrak{f})  \biggr).
\]
For \((\mathfrak n_1,\mathfrak n_2)=1\), substituting \eqref{eq:S shifted formula} into \eqref{eq:S formula for B main} gives
\begin{equation}\label{eq:complete-twisted-diagonal}
\mathcal B_{\mathrm{main}}(\mathfrak n_1,\mathfrak n_2) =
\frac{\gamma_{-1}(K)h_K^*(\mathfrak f)\phi(\mathfrak f)}
{\norm(\mathfrak f)\norm(\mathfrak n_1\mathfrak n_2)^{1/2}}
\mathcal L_0(\mathfrak n_1,\mathfrak n_2)
+ \mathsf{O}_K(h_K^*(\mathfrak f)d(\mathfrak f)
\norm(\mathfrak f)^{-1/2}),
\end{equation}
where
\[
\mathcal L_0(\mathfrak  n_1,\mathfrak  n_2)
=\log\frac{|d_K|\norm(\mathfrak f)}
{4\pi^2\norm(\mathfrak n_1\mathfrak n_2)}
+2 \frac{\Gamma'}{\Gamma}\biggl(\frac{1}{2}\biggr)
+2\frac{\gamma_0(K)}{\gamma_{-1}(K)}
+2\eta(\mathfrak f).
\]
\subsection{Averaged off-diagonal bound}\label{sec:twisted-off-diagonal}

The off-diagonal contribution of \eqref{second moments} comes from the terms with \( \mathfrak{a}_1 \mathfrak{m}_1 \neq \mathfrak{a}_2 \mathfrak{m}_2 \). The contribution is \(\ll \beta(\mathfrak{m}_1, \mathfrak{m}_2) \)
where
\[
	\beta(\mathfrak{m}_1, \mathfrak{m}_2) \coloneqq \sum_{n=1}^\infty\left|W\left(\frac{4n\pi^2}{|d_{K}|\norm(\f)}\right)\right|n^{\frac{-1}{2}}\sum_{\norm(\a_1\a_2)=n}\ssum_{1 \neq \mathfrak{a}_1 \mathfrak{m}_1 (\mathfrak{a}_2 \mathfrak{m}_2)^{-1} \in P_{K}(\t) \\ \t \mid \f}h_{K}(\t)
\]
Following Iwaniec and Sarnak \cite{Iwaniec-Sarnak}, we bound the average of \(\beta(\mathfrak m_1,\mathfrak m_2)\) over \(\norm(\mathfrak m_1),\norm(\mathfrak m_2)\le M\), with weight \(\norm(\mathfrak m_1\mathfrak m_2)^{-1/2}\). This suffices for the mollified second moment in \Cref{sec:mollified-second-moment}. Pulling out the sum over \(\mathfrak t\) and combining the other sums by setting \(\mathfrak J_i=\mathfrak a_i\mathfrak m_i\), for \(i=1,2\), gives
\begin{multline*}
\sum_{\norm(\mathfrak m_1),\norm(\mathfrak m_2)\le M}
\norm(\mathfrak m_1\mathfrak m_2)^{-1/2}
\beta(\mathfrak m_1,\mathfrak m_2)\\
\le
\sum_{\mathfrak t\mid\mathfrak f}h_K(\mathfrak t)
\ssum_{1\ne\mathfrak J_1\mathfrak J_2^{-1}\in P_K(\mathfrak t)}
\norm(\mathfrak J_1\mathfrak J_2)^{-1/2}
\ssum_{\norm(\mathfrak m_1),\norm(\mathfrak m_2)\le M\\
\mathfrak m_i\mid\mathfrak J_i}
	\left|W\!\left(
	\frac{4\pi^2}{|d_K|}
	\frac{\norm(\mathfrak J_1\mathfrak J_2)}
	{\norm(\mathfrak m_1\mathfrak m_2\mathfrak f)}
	\right)\right|.
\end{multline*}
The number of decompositions \(\mathfrak J_i=\mathfrak a_i\mathfrak m_i\), with \(\norm(\mathfrak m_i)\le M\), is at most \(d(\mathfrak J_i)\ll_{K,\varepsilon}\norm(\mathfrak J_i)^\varepsilon\). Moreover, the rapid decay of \(W\) and the bound \(\norm(\mathfrak m_1\mathfrak m_2)\le M^2\) give
\[
\left|W\!\left(\frac{4\pi^2}{|d_K|}
\frac{\norm(\mathfrak J_1\mathfrak J_2)}
{\norm(\mathfrak m_1\mathfrak m_2\mathfrak f)}\right)\right|
\ll_K \left(1+\frac{\norm(\mathfrak J_1\mathfrak J_2)}{M^2\norm(\mathfrak f)}\right)^{-5}.
\]
Consequently,
\begin{multline*}
\sum_{\norm(\mathfrak m_1),\norm(\mathfrak m_2)\le M}
\norm(\mathfrak m_1\mathfrak m_2)^{-1/2}\beta(\mathfrak m_1,\mathfrak m_2)\\
\ll_{K,\varepsilon}\sum_{\mathfrak t\mid\mathfrak f}h_K(\mathfrak t)
\ssum_{1\ne\mathfrak J_1\mathfrak J_2^{-1}\in P_K(\mathfrak t)}
\norm(\mathfrak J_1\mathfrak J_2)^{-1/2+\varepsilon}
\left(1+\frac{\norm(\mathfrak J_1\mathfrak J_2)}{M^2\norm(\mathfrak f)}\right)^{-5}.
\end{multline*}
Apply Lemma~\ref{lem:ray-ideal-pair-lattice} with \(X=M\norm(\mathfrak f)^{1/2}\). The inner sum is
\(\ll_{K,\varepsilon}\norm(\mathfrak t)^{-1}(M\norm(\mathfrak f)^{1/2})^{1+2\varepsilon}\log(2M\norm(\mathfrak f)^{1/2})\).
Since \(h_K(\mathfrak t)/\norm(\mathfrak t)\ll_K1\), we may absorb the logarithm and \(d(\mathfrak f)\) by increasing and then renaming \(\varepsilon\), obtaining
\begin{equation} \label{eq:twisted second moment off diagonal averaged bound}
	\sum_{\norm(\mathfrak m_1),\norm(\mathfrak m_2)\le M}  \norm (\mathfrak{m}_1 \mathfrak{m}_2)^{-1/2} \beta(\mathfrak{m}_1, \mathfrak{m}_2)  \ll M^{1+\varepsilon} \norm(\mathfrak{f})^{1/2+\varepsilon}.
\end{equation}

\section{Mollified second moment}\label{sec:mollified-second-moment}
In this section we prove Proposition~\ref{prop:second mollified moment}. The expansion of the second mollified moment is
\[
    \begin{aligned}
    \mathcal D(M)
    \coloneqq \sideset{}{^*}\sum_{\chi\bmod\mathfrak f}
    \left|\mathcal M(\chi)L\!\left(\frac12,\chi\right)\right|^2
    = \sum_{\substack{\norm(\mathfrak m_1),\norm(\mathfrak m_2)\le M \\
    (\mathfrak m_1\mathfrak m_2,\mathfrak f)=1}}
    \frac{\widetilde\mu(\mathfrak m_1)\widetilde\mu(\mathfrak m_2)}{\norm(\mathfrak m_1\mathfrak m_2)^{1/2}}
    \mathcal B(\mathfrak m_1,\mathfrak m_2).
    \end{aligned}
\]
The error term in \eqref{eq:complete-twisted-diagonal} averaged over the mollifier is
\begin{equation*}
\begin{aligned}
\ll h_K^*(\mathfrak f)d(\mathfrak f)
\norm(\mathfrak f)^{-1/2}
\sum_{\substack{\norm(\mathfrak m_1)\le M,\ \norm(\mathfrak m_2)\le M\\
(\mathfrak m_1\mathfrak m_2,\mathfrak f)=1}}
\frac{1}{\norm(\mathfrak m_1\mathfrak m_2)^{1/2}}
&\le
h_K^*(\mathfrak f)d(\mathfrak f)
\norm(\mathfrak f)^{-1/2}
\left(
\sum_{\norm(\mathfrak m)\le M}
\frac{1}{\norm(\mathfrak m)^{1/2}}
\right)^2\\
&\ll
M\norm(\mathfrak f)^{1/2+\varepsilon}.
\end{aligned}
\end{equation*}
The summation in parentheses is evaluated by partial summation on \eqref{eq:ideal-density} of Lemma~\ref{lem:ideal-density-mertens}.
This error may be absorbed into the error term of \eqref{eq:twisted second moment off diagonal averaged bound}. We thus obtain
\begin{equation}\label{eq:mollified-second-moment-decomposition}
\mathcal D(M)
=\mathcal D_{\mathrm{diag}}(M) + \mathsf{O}(M^{1+\varepsilon} \norm(\mathfrak{f})^{1/2+\varepsilon})
\end{equation}
where \(\mathcal D_{\mathrm{diag}}(M)\) is the residue term in \eqref{eq:complete-twisted-diagonal}, summed against the mollifier coefficients. To write it explicitly, we set \(\mathfrak l=(\mathfrak m_1,\mathfrak m_2)\) and \(\mathfrak m_i=\mathfrak l\mathfrak n_i\) and expand the coprimality condition using \(1_{(\mathfrak n_1,\mathfrak n_2)=1}=\sum_{\mathfrak d\mid(\mathfrak n_1,\mathfrak n_2)}\mu(\mathfrak d)\). After replacing each \(\mathfrak n_i\) by \(\mathfrak d\mathfrak n_i\), relabel \(\mathfrak l\mathfrak d\) as \(\mathfrak l\). Then \(\mathfrak d\mid\mathfrak l\), and the norm factor becomes \(1/(\norm(\mathfrak l)\norm(\mathfrak d)\norm(\mathfrak n_1\mathfrak n_2))\). The factor \(\mu^2(\mathfrak l)\) may be inserted because the mollifier coefficients vanish unless \(\mathfrak l\) is squarefree. Thus
\[
\mathcal D_{\mathrm{diag}}(M)
=
\frac{\gamma_{-1}(K)h_K^*(\f)\phi(\f)}{\norm(\f)}
\sum_{\substack{\norm(\mathfrak l)\le M\\(\mathfrak l,\f)=1}}
\frac{\mu^2(\mathfrak l)}{\norm(\mathfrak l)}
\sum_{\mathfrak d\mid\mathfrak l}
\frac{\mu(\mathfrak d)}{\norm(\mathfrak d)}
\sum_{\substack{\mathfrak n_1,\mathfrak n_2\\
(\mathfrak n_1\mathfrak n_2,\mathfrak f)=1}}
\frac{\widetilde{\mu}(\mathfrak l\mathfrak n_1)
\widetilde{\mu}(\mathfrak l\mathfrak n_2)}
{\norm(\mathfrak n_1)\norm(\mathfrak n_2)}\mathcal{L}_0(\mathfrak{d}\mathfrak{n}_1, \mathfrak{d}\mathfrak{n}_2).
\]
We write
\[
\mathcal L_0(\mathfrak{d}\mathfrak{n}_1, \mathfrak{d}\mathfrak{n}_2) = \log\frac{\norm(\mathfrak f)}{\norm(\mathfrak d\mathfrak n_1)}
+\log\frac{\norm(\mathfrak f)}{\norm(\mathfrak d\mathfrak n_2)}
-\log\norm(\mathfrak f) + R(K,\mathfrak f)
\]
where
\(
R(K,\mathfrak f)=
\log\frac{|d_K|}{4\pi^2}
+2\frac{\Gamma'}{\Gamma}\!\left(\frac12\right)
+2\frac{\gamma_0(K)}{\gamma_{-1}(K)}
+2\eta(\mathfrak f) \) is \(\ll_K\log_2\norm(\mathfrak f)\), by Lemma~\ref{lem:uniform Mertens}. In the notation of Corollary~\ref{lem:T0-T1-evaluation}, the innermost sum is
\begin{equation} \label{eq:innermost sum using Lemma for T0-T1}
	\sum_{\substack{\mathfrak n_1,\mathfrak n_2\\
	(\mathfrak n_1\mathfrak n_2,\mathfrak f)=1}}
	\frac{\widetilde{\mu}(\mathfrak l\mathfrak n_1)
	\widetilde{\mu}(\mathfrak l\mathfrak n_2)}
	{\norm(\mathfrak n_1)\norm(\mathfrak n_2)}\mathcal{L}_0(\mathfrak{d}\mathfrak{n}_1, \mathfrak{d}\mathfrak{n}_2)
	=
		2T_0(\mathfrak l)T_1(\mathfrak l,\mathfrak d)
		+ (-\log \norm(\f)  +R(K, \mathfrak{f})) T_0(\mathfrak l)^2.
\end{equation}
By the same corollary, we have
\[
	\begin{aligned}
		T_0(\mathfrak{l})
		&= T_{0,\mathrm{main}}(\mathfrak{l}) + \mathcal{O}(\mathcal{E}_1(\mathfrak{l})), \\
		T_1(\mathfrak{l},\mathfrak{d})
		&= T_{0,\mathrm{main}}(\mathfrak{l}) (\log (M) + \log \norm(\mathfrak{f}) - \log \norm(\mathfrak{l}) - \log \norm(\mathfrak{d})) + \mathcal{O}(\log (M) \mathcal{E}_1(\mathfrak{l})),
	\end{aligned}
\]
where \( T_{0,\mathrm{main}}(\mathfrak{l}) = \frac{1}{\gamma_{-1}(K)} \frac{\norm(\mathfrak{f})}{\phi(\mathfrak{f})} \frac{\norm(\mathfrak{l})}{\phi(\mathfrak{l})} \frac{\mu(\mathfrak{l})}{\log(M)}  \).
Summing \eqref{eq:innermost sum using Lemma for T0-T1} over \( \mathfrak{l} \) thus yields
\begin{multline}  \label{eq:expansion of D diag}
	\mathcal{D}_{\mathrm{diag}} (M) = \frac{\gamma_{-1}(K) h_K^*(\mathfrak{f}) \phi(\mathfrak{f})}{\norm(\mathfrak{f})} \biggl[(\log \norm(\mathfrak{f}) + R(K,\mathfrak{f})) \mathcal{D}_{\mathrm{diag},1}(M) + 2 \mathcal{D}_{\mathrm{diag},2}(M) - 2 \mathcal{D}_{\mathrm{diag},3}(M) \\
	+ \mathsf{O}(\mathcal{E}_{\mathrm{diag}}(M)) \biggr],
\end{multline}
where
\[
	\begin{aligned}
    \mathcal{D}_{\mathrm{diag},1}(M)
	&= \ssum_{\norm(\mathfrak{l}) \le M \\ (\mathfrak{l}, \mathfrak{f}) = 1} \frac{\mu^2(\mathfrak{l})}{\norm(\mathfrak{l})} T_0(\mathfrak{l})^2 \sum_{\mathfrak{d} \mid \mathfrak{l}} \frac{\mu(\mathfrak{d})}{\norm(\mathfrak{d})}, \\
	\mathcal{D}_{\mathrm{diag},2}(M)
	&= \ssum_{\norm(\mathfrak{l}) \le M \\ (\mathfrak{l}, \mathfrak{f}) = 1} \frac{\mu^2(\mathfrak{l})}{\norm(\mathfrak{l})} T_0(\mathfrak{l})^2 \log \frac{M}{\norm(\mathfrak{l})} \sum_{\mathfrak{d} \mid \mathfrak{l}} \frac{\mu(\mathfrak{d})}{\norm(\mathfrak{d})}, \\
	\mathcal{D}_{\mathrm{diag},3}(M)
	&= \ssum_{\norm(\mathfrak{l}) \le M \\ (\mathfrak{l}, \mathfrak{f}) = 1} \frac{\mu^2(\mathfrak{l})}{\norm(\mathfrak{l})} T_0(\mathfrak{l})^2  \sum_{\mathfrak{d} \mid \mathfrak{l}} \frac{\mu(\mathfrak{d})}{\norm(\mathfrak{d})}\log \norm(\mathfrak{d}),
	\end{aligned}
\]
and
\begin{equation} \label{eq:E diag formula}
	\mathcal{E}_{\mathrm{diag}}(M)
	= \log(M) \ssum_{\norm(\mathfrak{l}) \le M \\ (\mathfrak{l}, \mathfrak{f}) = 1} \frac{\mu^2(\mathfrak{l})}{\norm(\mathfrak{l})} \biggl(\frac{\norm(\mathfrak{f})}{\phi(\mathfrak{f})} \times \frac{\norm(\mathfrak{l})}{\phi(\mathfrak{l}) \log(M)}  + \mathcal{E}_1(\mathfrak{l})\biggr) \mathcal{E}_1(\mathfrak{l}) \sum_{\mathfrak{d} \mid \mathfrak{l}} \frac{1}{\norm(\mathfrak{d})}.
\end{equation}
Here we note that \( \log \norm(\mathfrak{l}) \) and \( \log \norm(\mathfrak{d}) \) are both \( \ll \log(M) \), so the contribution from \( T_0(\mathfrak{l})^2 - T_{0,\mathrm{main}}(\mathfrak{l})^2 \) in each of the \( \mathcal{D}_{\mathrm{diag},i}(M) \) can be absorbed into \eqref{eq:E diag formula}. So \eqref{eq:expansion of D diag} holds when we replace \( T_0(\mathfrak{l})^2 \) with \( T_{0,\mathrm{main}}(\mathfrak{l})^2 \) in each of the \( \mathcal{D}_{\mathrm{diag},i}(M) \). After this replacement we resolve the sums over \( \mathfrak{d} \) through the identities \( \sum_{\mathfrak{d} \mid \mathfrak{l}} \frac{\mu(\mathfrak{d})}{\norm(\mathfrak{d})} = \frac{\phi(\mathfrak{l})}{\norm(\mathfrak{l})} \), \( -\sum_{\mathfrak{d} \mid \mathfrak{l}}\frac{\mu(\mathfrak{d})}{\norm(\mathfrak{d})} \log \norm(\mathfrak{d}) = \frac{\phi(\mathfrak{l})}{\norm(\mathfrak{l})} \eta(\mathfrak{l}) \) that hold by virtue of \( \mathfrak{l} \) being squarefree. In summary, we have the simplified expansion
\begin{multline} \label{eq:D diag simplified expansion}
\mathcal{D}_{\mathrm{diag}}(M) = \frac{h_K^*(\mathfrak{f})\norm(\mathfrak{f})}{\gamma_{-1}(K) \phi(\mathfrak{f}) \log^2(M)} \biggl[ (\log \norm(\mathfrak{f}) + R(K,\mathfrak{f})) \mathcal{D}_{\mathrm{diag},1}'(M) + 2 \mathcal{D}_{\mathrm{diag},2}'(M) - 2 \mathcal{D}_{\mathrm{diag},3}'(M)\biggr] \\ + \mathsf{O} \biggl(\frac{h_K^*(\mathfrak{f}) \phi(\mathfrak{f}) \mathcal{E}_{\mathrm{diag}}(M)}{\norm(\mathfrak{f})}\biggr),
\end{multline}
where
\[
	\begin{gathered}
    \mathcal{D}_{\mathrm{diag},1}'(M)
	=  \ssum_{\norm(\mathfrak{l}) \le M \\ (\mathfrak{l}, \mathfrak{f}) = 1} \frac{\mu^2(\mathfrak{l})}{\norm(\mathfrak{l})} \frac{\norm(\mathfrak{l})}{\phi(\mathfrak{l})}, \qquad
	\mathcal{D}_{\mathrm{diag},2}'(M) = \ssum_{\norm(\mathfrak{l}) \le M \\ (\mathfrak{l}, \mathfrak{f}) = 1} \frac{\mu^2(\mathfrak{l})}{\norm(\mathfrak{l})} \frac{\norm(\mathfrak{l})}{\phi(\mathfrak{l})}\log \frac{M}{\norm(\mathfrak{l})}, \\
	\mathcal{D}_{\mathrm{diag},3}'(M)
	= -\ssum_{\norm(\mathfrak{l}) \le M \\ (\mathfrak{l}, \mathfrak{f}) = 1} \frac{\mu^2(\mathfrak{l})}{\norm(\mathfrak{l})} \frac{\norm(\mathfrak{l})}{\phi(\mathfrak{l})}\eta(\mathfrak{l}).
	\end{gathered}
\]
For \( \mathcal{D}_{\mathrm{diag},3}'(M) \), we expand \( \eta(\mathfrak{l}) = \sum_{\mathfrak{p} \mid \mathfrak{l}}  \frac{\log \norm(\mathfrak{p})}{\norm(\mathfrak{p}) - 1}\) to write it as
\[
	\begin{aligned}
		-\mathcal{D}_{\mathrm{diag},3}'(M)
		&= \ssum_{\norm(\mathfrak{p}) \le M \\ (\mathfrak{p},\mathfrak{f}) = 1} \frac{\log \norm(\mathfrak{p})}{\norm(\mathfrak{p}) - 1} \ssum_{\norm(\mathfrak{l}) \le M/\norm(\mathfrak{p}) \\ (\mathfrak{l}, \mathfrak{f}) = 1} \frac{\mu^2(\mathfrak{p}\mathfrak{l})}{\phi(\mathfrak{p}\mathfrak{l})} \\
		&= \ssum_{\norm(\mathfrak{p}) \le M \\ (\mathfrak{p},\mathfrak{f}) = 1} \frac{\log \norm(\mathfrak{p})}{(\norm(\mathfrak{p}) - 1)^2}
		\ssum_{\norm(\mathfrak{l}) \le M/\norm(\mathfrak{p}) \\ (\mathfrak{l}, \mathfrak{pf}) = 1} \frac{\mu^2(\mathfrak{l})}{\norm(\mathfrak{l})} \frac{\norm(\mathfrak{l})}{\phi(\mathfrak{l})}
	\end{aligned}
\]
Lemma~\ref{conrey2} applies to \(\mathcal D_{\mathrm{diag},i}'(M)\), for \(i=1,2\), with \(j=0,1\) and \(g(\mathfrak l)=\mu^2(\mathfrak l)\norm(\mathfrak l)/\phi(\mathfrak l)\) for \((\mathfrak{l}, \mathfrak f) = 1\). Its conductor condition holds with \(B=1/\theta\), since \(\norm(\mathfrak f)=M^{1/\theta}\). To bound the inner sum in \(-\mathcal D_{\mathrm{diag},3}'(M)\) uniformly in \(\mathfrak p\), drop the coprimality condition. The identity \(\norm(\mathfrak l)/\phi(\mathfrak l)=\sum_{\mathfrak d\mid\mathfrak l}\mu^2(\mathfrak d)/\phi(\mathfrak d)\) and the ideal-counting estimate give, for \(x\ge1\),
\[
\sum_{\norm(\mathfrak l)\le x}\frac{\mu^2(\mathfrak l)}{\phi(\mathfrak l)}
\le\sum_{\norm(\mathfrak d)\le x}\frac{\mu^2(\mathfrak d)}{\norm(\mathfrak d)\phi(\mathfrak d)}
\sum_{\norm(\mathfrak a)\le x/\norm(\mathfrak d)}\frac1{\norm(\mathfrak a)}
\ll_K\log(2x).
\]
The sum over \(\mathfrak{d}\) evidently converges. The outer sum over \(\mathfrak{p}\) also converges, so we deduce that \(\mathcal D_{\mathrm{diag},3}'(M) \ll \log(M) \). For the first two sums, Lemma~\ref{conrey2} gives
\[
	\begin{aligned}
		\mathcal{D}_{\mathrm{diag},1}'(M)
		&= \gamma_{-1}(K) \frac{\phi(\mathfrak{f})}{\norm(\mathfrak{f})} \log(M) + D(\mathfrak{f}) + \mathsf{O}(1), \\
		\mathcal{D}_{\mathrm{diag},2}'(M)
		&= \gamma_{-1}(K) \frac{\phi(\mathfrak{f})}{\norm(\mathfrak{f})} \frac{\log^2(M)}{2} + D(\mathfrak{f}) \log(M) + \mathsf{O}(\log(M)),
	\end{aligned}
\]
where
\[
	D(\mathfrak{f}) = \frac{\phi(\mathfrak{f})}{\norm(\mathfrak{f})} \biggl[\gamma_0(K) + \gamma_{-1}(K)  \biggl(\eta(\mathfrak{f}) + \sum_{\mathfrak{p} \nmid \mathfrak{f}} \frac{\log \norm(\mathfrak{p})}{\norm(\mathfrak{p}) (\norm(\mathfrak{p}) - 1)}  \biggr) \biggr].
\]
We put these estimates into \eqref{eq:D diag simplified expansion}. Upon noting that \( D(\mathfrak{f}) \ll \log_2(M) \), it becomes
\[
	\begin{aligned}
	\mathcal{D}_{\mathrm{diag}}(M)
		&=  h_K^*(\mathfrak{f}) \biggl[ \frac{\log \norm(\mathfrak{f})}{\log(M)} + 1 + \mathsf{O} \biggl(\frac{\log_2^2(M)}{\log(M)}\biggr)  \biggr] + \mathsf{O} \biggl(\frac{h_K^*(\mathfrak{f}) \phi(\mathfrak{f}) \mathcal{E}_{\mathrm{diag}}(M)}{\norm(\mathfrak{f})}\biggr) \\
		&=  h_K^*(\mathfrak{f}) \biggl[ 1 + \frac{1}{\theta} + \mathsf{O} \biggl(\frac{\log_2^2(M)}{\log(M)}\biggr)  \biggr] + \mathsf{O} \biggl(\frac{h_K^*(\mathfrak{f}) \phi(\mathfrak{f}) \mathcal{E}_{\mathrm{diag}}(M)}{\norm(\mathfrak{f})}\biggr),
	\end{aligned}
\]
where \( \frac{1}{\theta} \) arises from \(M=\norm(\mathfrak f)^\theta\). It remains to bound \(\mathcal E_{\mathrm{diag}}(M)\). In \eqref{eq:E diag formula}, drop the squarefree and coprimality conditions and use
\(\sum_{\mathfrak d\mid\mathfrak l}\norm(\mathfrak d)^{-1}\le\norm(\mathfrak l)/\phi(\mathfrak l)\ll_K\log_2(M)\).
The resulting estimate is
\[
\frac{\phi(\mathfrak f)}{\norm(\mathfrak f)}\mathcal E_{\mathrm{diag}}(M)
\ll_{K,\theta}\log_2^2(M)\mathcal E_{\mathrm{diag},1}(M)
+\log(M)\log_2(M)\mathcal E_{\mathrm{diag},2}(M),
\]
where, for \(i=1,2\), we write
\(\mathcal E_{\mathrm{diag},i}(M)=\sum_{\norm(\mathfrak l)\le M}\mathcal E_1(\mathfrak l)^i/\norm(\mathfrak l)\).
Substitute the formula
\[
\mathcal E_1(\mathfrak l)=\frac{\log_2^2(M)}{\log^2(M)}\bigl(1+\log(M)\log_2(M)(\norm(\mathfrak l)/M)^{1/(C\log_2 M)}\bigr),
\]
with \(C=C(K,\theta)>0\) as in Lemma~\ref{lem:conrey-1}. Partial summation using \eqref{eq:ideal-density} gives
\[
\begin{aligned}
\sum_{\norm(\mathfrak l)\le M}\frac1{\norm(\mathfrak l)}&\ll \log(M) ,\\
\sum_{\norm(\mathfrak l)\le M}\frac{(\norm(\mathfrak l)/M)^{j/(C\log_2 M)}}{\norm(\mathfrak l)} &\ll \log_2(M), \qquad j=1,2.
\end{aligned}
\]
Expand \(\mathcal E_1(\mathfrak l)\) for the first estimate and bound its square by twice the sum of the squares of its two terms for the second. We obtain
\[
\begin{aligned}
\mathcal E_{\mathrm{diag},1}(M)
&\ll_{K,\theta}\frac{\log_2^2 M}{\log M}
+\frac{\log_2^4 M}{\log M}
\ll_{K,\theta}\frac{\log_2^4 M}{\log M},\\
\mathcal E_{\mathrm{diag},2}(M)
&\ll_{K,\theta}\frac{\log_2^4 M}{\log^3 M}
+\frac{\log_2^7 M}{\log^2 M}
\ll_{K,\theta}\frac{\log_2^7 M}{\log^2 M}.
\end{aligned}
\]
Consequently, \(\phi(\mathfrak f)\mathcal E_{\mathrm{diag}}(M)/\norm(\mathfrak f)\ll_{K,\theta}\log_2^8(M)/\log M\). Combining this with \eqref{eq:mollified-second-moment-decomposition} proves Proposition~\ref{prop:second mollified moment} in the more precise form
\begin{equation} \label{D asymptotic second moment}
\mathcal{D}(M) = h_K^*(\mathfrak{f}) \biggl[1 + \frac{1}{\theta} + \mathsf{O} \biggl(\frac{\log_2^8(M)}{\log(M)}\biggr)\biggr] + \mathsf{O}(M^{1+\varepsilon} \norm(\mathfrak{f})^{1/2+\varepsilon}).
\end{equation}

\section{Proof of Theorem 1.1} \label{sec:main thm}

We now prove \Cref{thm:main}.  Let
\[
\mathcal N(\mathfrak f)
:=
\#\left\{
\chi\bmod\mathfrak f:
\chi \text{ primitive},\
L\!\left(\frac12,\chi\right)\ne0
\right\}.
\]
By Lemma~\ref{lem:primitive-family-size}, if a prime ideal of norm \(2\)
divides \(\mathfrak f\) exactly once, then the primitive family is empty
and the asserted lower bound is trivial.  We may therefore assume that
this does not occur.  The same lemma then gives
\(h_K^*(\mathfrak f)=\norm(\mathfrak f)^{1-\littleo(1)}\).
By \eqref{C asymptotic first moment}, for \(M=\norm(\mathfrak f)^\theta\)
with \(0<\theta<1/2\), we have
\[
\mathcal C(M)
=
(1+\littleo(1))h_K^*(\mathfrak f).
\]
Indeed, the error terms in \eqref{C asymptotic first moment} are
\[
\bigO\!\left(
\norm(\mathfrak f)^{3\theta/2+1/4+\varepsilon}
+
\norm(\mathfrak f)^{\theta/2+3/4+\varepsilon}
\right),
\]
which are \(\littleo(h_K^*(\mathfrak f))\) for \(\theta<1/2\), by
Lemma~\ref{lem:primitive-family-size} after choosing \(\varepsilon>0\)
sufficiently small.

Similarly, by \eqref{D asymptotic second moment},
\[
\mathcal D(M)
=
\left(1+\frac1\theta+\littleo(1)\right)h_K^*(\mathfrak f).
\]
The off-diagonal error is \(\littleo(h_K^*(\mathfrak f))\) in the same range
\(0<\theta<1/2\), again by
Lemma~\ref{lem:primitive-family-size}.

The Cauchy-Schwarz inequality reads \(|\mathcal C(M)|^2\le\mathcal N(\mathfrak f)\mathcal D(M)\), or
\[
\mathcal N(\mathfrak f)
\ge
\frac{|\mathcal C(M)|^2}{\mathcal D(M)}
=
\left(
\frac{\theta}{1+\theta}+\littleo(1)
\right)
h_K^*(\mathfrak f).
\]
For each fixed \(0<\theta<1/2\), first let \(\norm(\mathfrak f)\to\infty\)
over nonempty primitive families. The preceding inequality gives
\[
\liminf_{\substack{\norm(\mathfrak f)\to\infty\\h_K^*(\mathfrak f)>0}}
\frac{\mathcal N(\mathfrak f)}{h_K^*(\mathfrak f)}
\ge\frac{\theta}{1+\theta}.
\]
Taking the supremum over the fixed values \(0<\theta<1/2\) yields
the lower bound \(1/3\). This proves
\Cref{thm:main}.

\section*{Acknowledgments}
The authors wish to thank Alia Hamieh for the many helpful discussions throughout this project. The second author gratefully acknowledges support from the Mitacs Globalink program during the initial stages.

\bibliography{Hecke}

\bigskip
\noindent\begin{minipage}{\textwidth}
\footnotesize
\textsc{Thurman Ye}\\
Department of Mathematics, University of California, Irvine\\
Irvine, CA 92697-3875, USA\\
\textit{Email address:} \email{thurmay@uci.edu}
\end{minipage}

\medskip
\noindent\begin{minipage}{\textwidth}
\footnotesize
\textsc{Xu Zhuang}\\
Department of Mathematics, University of California, Irvine\\
Irvine, CA 92697-3875, USA\\
\textit{Email address:} \email{xzhuang8@uci.edu}
\end{minipage}
\end{document}